\documentclass[twocolumn]{autart}    
\usepackage{graphicx}     
\usepackage{subcaption}  
\usepackage{amssymb}
\usepackage{amsmath}  

\usepackage{algorithmic}
\usepackage{algorithm}
\usepackage[round,authoryear]{natbib}

\allowdisplaybreaks[3]

\begin{document}

\begin{frontmatter}

\title{Inverse reinforcement learning for indefinite mean-field social optimization with multiplicative noise} 

\thanks[footnoteinfo]{Corresponding author Bing-Chang Wang. }
\author[Paestum]{Ying Cao}\ead{yingcao@mail.sdu.edu.cn},    
\author[Rome]{Xun Li}\ead{li.xun@polyu.edu.hk},  
\author[Paestum]{Bing-Chang Wang}\ead{bcwang@sdu.edu.cn}

\address[Paestum]{School of Control Science and Engineering, Shandong University, Jinan 250061, China}  
\address[Rome]{Department of Applied Mathematics, The Hong Kong Polytechnic University, Hong Kong, China}             


\begin{abstract}                          
This paper studies the inverse reinforcement learning (RL) problem for linear-quadratic mean-field (MF) social optimization. The considered system features multiplicative noise and indefinite cost weights, which violate standard convexity assumptions and pose analytical challenges. The goal is to recover unknown social cost weights from expert demonstrations and reproduce the optimal control policies. This requires solving coupled stochastic algebraic Riccati equations and Lyapunov equations with unknown system dynamics. To this end, we first propose a model-based inverse RL algorithm with two sequential loops that separately handle individual and MF dynamics, and we prove its convergence and closed-loop stabilizability. Moreover, we characterize the non-uniqueness of the recovered cost weights. To eliminate reliance on system dynamics, we develop a model-free inverse RL algorithm using integral RL and least-squares identification, which requires only measured trajectory data satisfying mild rank conditions. Finally, numerical simulations validate the effectiveness of the proposed approaches.
\end{abstract}

\end{frontmatter}

\section{Introduction}
Large-scale multi-agent systems have attracted significant attention in both control theory and artificial intelligence. As the number of agents grows, the complexity of pairwise interactions increases dramatically, rendering conventional decentralized modeling and control computationally intractable. Mean-field (MF) games and control overcome this challenge by approximating the collective influence of all agents through an MF term, thereby decoupling the interactions and avoiding the curse of dimensionality \citep{Lasry2007,Bensoussan2013}. They have been widely applied in diverse fields, including smart grids \citep{Ma2013,Huang2025}, economics \citep{Wang2019}, mobile communications \citep{WangY2024,Gu2026}, and biological networks \citep{Petrakova2025}.

MF games describe non-cooperative decision-making, where agents optimize individual objectives \citep{Huang2007}. MF social control, by contrast, pursues social optima in large-population cooperative systems, which form a class of team decision problems \citep{Ho1980}. The problem aims to minimize the social cost, defined as the sum of all agents' individual costs, with the linear-quadratic (LQ) formulation receiving significant attention due to its analytical tractability. In \cite{Huang2012}, the authors proposed the social certainty equivalence approach to asymptotically attain social optima in LQ MF control problems. Following this research direction, a series of extensions have been developed for volatility uncertainties \citep{Huang2021}, input constraints \citep{Du2022}, Markov jump parameters \citep{Wang2017}, and model uncertainties \citep{WangBC2021}. Although these methods perform well in ideal settings, they have two major limitations. First, they rely on full knowledge of the system dynamics, which restricts their applicability in real-world scenarios. Second, all of them require manual design of cost functions, which relies heavily on subjective judgment and practical experience, making it challenging to construct targeted cost functions in complex environments.

To eliminate the reliance on explicit system models, reinforcement learning (RL) \citep{Sutton2018} has been employed to develop data-driven methods for addressing MF games and social optimization problems. \cite{Xu2023} used a policy iteration (PI) algorithm enhanced with integral RL to solve continuous-time LQ MF games in a completely model-free setting. \cite{WangBC2024} developed an online value iteration algorithm for LQ MF social control with ergodic cost functions. \cite{Xu2025} proposed a model-free approach using data-driven PI that bypasses the traditional double-loop structure. These works demonstrate the effectiveness of RL in learning optimal policies from data. 
Despite the model-free advantage of RL, the cost function still needs to be predefined. In practice, constructing a cost function that yields the desired optimal behavior is a non-trivial task, and often requires substantial domain knowledge. This limitation naturally motivates alternative paradigms that can directly infer latent objectives from expert demonstrations.

Originating from \cite{Andrew2000}, inverse RL adopts a data-driven scheme to recover cost functions from observed expert behaviors, thereby eliminating the need for manual construction. For instance, in autonomous driving, the vehicle infers reward functions from expert trajectories to mimic human driving behavior. Inverse RL has found extensive applications across various domains, including intelligent transportation \citep{You2019,Zhao2026}, smart grids \citep{Chen2024}, aerospace \citep{Luke2026}, and economics \citep{Sun2023}. In the context of optimal control and games, inverse RL provides a data-driven counterpart to inverse optimal control (IOC), inferring stabilizing cost functions without prior knowledge of the system dynamics. \cite{Xue2022} proposed an inverse RL algorithm for tracking control based on IOC. \cite{Xue2023} developed an inverse RL approach to solve the behavior imitation problem for discrete-time systems. \cite{Lian2024} proposed a data-driven inverse RL algorithm for nonzero-sum multiplayer games. 

Most existing inverse RL studies focus on deterministic systems, while stochastic systems with multiplicative noise have received less attention. \cite{Sun2025} studied inverse RL for LQ stochastic optimal control. \cite{Chen2023} developed an adversarial inverse RL method for discrete-time finite-horizon MF noncooperative games. However, these studies have two limitations. First, existing MF inverse RL methods address only noncooperative games, whereas inverse RL for MF social optimization remains unexplored. Second, they consider only (semi-)positive definite cost weights, ignoring the more general indefinite-weight case. To the best of our knowledge, this paper is the first to fill this research gap, which has hindered behavioral analysis and cost function recovery for large-scale stochastic cooperative populations.

In this paper, we study an inverse RL problem for indefinite LQ MF social optimization with multiplicative noise. The objective is to infer unknown and possibly indefinite social cost weights from expert demonstrations, enabling a learner population to imitate the expert's optimal cooperative behavior without prior knowledge of the system dynamics. Unlike standard RL that performs policy optimization via trial and error for a given cost function, inverse RL addresses the dual problem of recovering the unknown objective that replicates the expert's optimal behavior. This involves solving an inverse mapping from policies to costs rather than a direct optimization from costs to policies.

This inverse RL problem poses significant technical challenges beyond the standard MF social control or RL setting. First, multiplicative noise that depends on both the state and control introduces additional unknown system information and leads to two stochastic algebraic Riccati equations (SAREs) in the expert's MF social control problem. Consequently, solving the inverse RL problem of reconstructing the unknown indefinite cost weights is associated with the two SAREs, a challenging problem that remains unaddressed in existing work. 
Second, the indefiniteness of weight matrices not only makes the social cost potentially unbounded below, rendering the MF social control problem ill-posed, but also invalidates the standard positive-definiteness assumptions that are crucial for stability guarantees in most inverse RL algorithms. Finally, when the system dynamics are unknown, a model-free approach needs to be developed for solving the inverse RL problem. 

To overcome these hurdles, we develop an inverse RL framework that covers both model-based and model-free settings. Starting from a model-based iterative algorithm, we address the difficulties caused by multiplicative noise and indefinite weights. Specifically, we observe that the second SARE involves two solution matrices, which would require solving a bivariate Lyapunov equation under a standard treatment. By subtracting the two SAREs to derive a new equation, we eliminate the terms associated with the diffusion coefficients. This not only reduces computational complexity but also simplifies the convergence and stability analysis. To tackle the indefiniteness issue, we introduce linear matrix inequality conditions to ensure the well-posedness of the problem and the existence of stabilizing solutions to the two SAREs. 
In the stability proof, unknown weight terms are eliminated by differencing the Lyapunov equation between consecutive iterations. Then, using an equivalent condition for mean-square stabilizability, closed-loop stability after the first iteration can be guaranteed under indefinite weights. In addition, an appropriate choice of the initial weights ensures the stabilizability of initial gains.
For the model-free design, we transform the model-based iterative equations into data-driven formulations via integral RL, and identify unknown parameters using the least-squares method. 

The main contributions of this work are summarized as follows. 
\begin{enumerate}
	\item[(1).] Under indefinite social cost weights and multiplicative noise, we develop a model-based inverse RL algorithm consisting of two sequential loops, which respectively focus on individual dynamics and MF dynamics. By iteratively solving the corresponding Lyapunov equation and SARE in each loop, the algorithm can converge to equivalent indefinite social cost weights for the learner, thereby learning the expert's optimal feedback gains. 
	\item[(2).] A rigorous convergence proof for the model-based inverse RL algorithm is provided, where the boundedness of iterative sequences is established by mathematical induction. A key result is that the feedback gains across all iterations ensure closed-loop stability under indefinite weights, which extends existing inverse RL stability theory beyond the conventional positive (semi-)definite case. 
	Moreover, we illustrate the non-uniqueness of the equivalent weights, and characterize all possible solutions to the inverse RL problem.
	\item[(3).] To eliminate the dependence on system dynamics, we develop a model-free inverse RL algorithm via integral RL, which requires only measured trajectory data for implementation. Furthermore, the convergence is established by showing its equivalence to the model-based RL algorithm under rank conditions.
\end{enumerate}

The remainder of this paper is organized as follows. In Section~\ref{1}, we formulate the indefinite MF social optimization, and further propose an inverse RL problem. Section~\ref{2} presents a model-based inverse RL algorithm and analyzes its theoretical properties. In Section~\ref{3}, we derive data-driven iterative equations and further develop a model-free inverse RL algorithm using only measured trajectory data. Section~\ref{4} provides numerical simulations to validate the effectiveness of the proposed approaches. Finally, Section~\ref{5} concludes this paper.

\textbf{Notation:} 
Let $\mathbb{R}$ be the set of real numbers. $\mathbb{R}^n$ is the $n$-dimensional Euclidean space, and $\mathbb{R}^{n\times m}$ is the set of all $n\times m$ real matrices. $\mathbb{S}^n$ denotes the set of all $n\times n$ real symmetric matrices.
For $v\in \mathbb{R}^n$ and $M\in \mathbb{R}^{n\times m}$, $v^{\mathrm{T}}$ and $M^{\mathrm{T}}$ denote their transposes, respectively; $\|\cdot\|$ denotes the Euclidean norm for $v$ and the Frobenius norm for $M$; $\|v\|_Q^2 = v^\mathrm{T}Qv$.  For a symmetric matrix $N\in\mathbb{S}^n$, $N>0$ ($N\ge 0$) indicates that $N$ is positive (semi-)definite. $\otimes$ denotes the Kronecker product. $\sigma(x(s),0\le s\le t)$ denotes the $\sigma$-algebra generated by the trajectory $\{x(s)\}_{0\le s\le t}$.

\section{Problem Formulation}\label{1}
In this section, we consider two large population systems with identical stochastic unknown dynamics for imitation learning: the expert population and the learner population. The expert solves an MF social control problem with indefinite social cost weights and generates optimal system trajectories. The learner has no access to the expert's social cost function or system dynamics, and seeks to learn the underlying social cost function solely from observed trajectories for imitating the expert's behavior.
\subsection{Mean-Field Social Control for Expert Population}\label{1.1}
Consider an expert population system with $N$ agents. The dynamics of the $i$th expert agent are described by the following stochastic differential equation (SDE):
\begin{align}\label{x_e}
	dx_{ei}(t) =& [Ax_{ei}(t)+Bu_{ei}(t)]dt \cr
	& + [Cx_{ei}(t)+Du_{ei}(t)]dw_i(t), \ t\ge 0,
\end{align}
where $x_{ei}(t)\in\mathbb{R}^n$ and $u_{ei}(t)\in\mathbb{R}^{m}$ are the state and control input (or policy as we shall see later) of an expert agent $i\in\mathcal{N}\triangleq\{1,2,\dots,N\}$, respectively.  $A, C\in\mathbb{R}^{n\times n}$ and $B, D\in\mathbb{R}^{n\times m}$ are unknown constant matrices. The noise processes $\{w_i(t), i\in \mathcal{N}\}$ are mutually independent standard one-dimensional Brownian motions, which are defined on a complete filtered probability space $(\Omega, \mathcal{F}, \left\lbrace \mathcal{F}_t\right\rbrace _{t\ge 0}, \mathbf{P})$. 

We denote by $x_e = \{x_{e1}, \dots , x_{eN}\}$ and $u_e = \{u_{e1}, \dots , \\u_{eN}\}$ the sets of states and control policies for the expert population, respectively. For an initial state set $x_e^0 = \{x_{e1}^0, \dots , x_{eN}^0\}$, we make the following assumption. 
\begin{assum}\label{a:initial_state}
	The initial states $\{x_{ei}^0,\,i\in \mathcal{N}\}$ are independent of the Brownian motions $\{w_i(t), i\in \mathcal{N}\}$, and are mutually independent with a common expectation $\mathbb{E}x_{ei}^0=\bar{x}_e^0$. In addition, there exists a finite constant $c_0$ such that $\max_{1\le i \le N}\mathbb{E}\|x_{ei}^0\|^2 < c_0$.
\end{assum}

The cost function for each expert agent is given by
\begin{align}\label{J_ei}
	J_{ei}(u_{e})=\mathbb{E} \left[ \int_{0}^{\infty}\big(\|x_{ei}-\Gamma_e x_{e}^{(N)}\|_{Q_e}^2+\|u_{ei}\|_{R_e}^2\big) d\tau\right], 
\end{align}
where the weight matrices $Q_e=Q_e^{\top}\in \mathbb{S}^{n}$, $R_e=R_e^{\top}\in \mathbb{S}^{m}$, and $\Gamma_e\in \mathbb{R}^{n\times n}$. Notably, $Q_e$ and $R_e$ are symmetric matrices that may be indefinite. Here, $x_{e}^{(N)}=\frac{1}{N}\sum_{j=1}^{N}x_{ej}$ denotes the average state of the expert population. The social cost function for the expert population is further defined as
\begin{align}\label{J_e}
	J_e^{soc}(u_{e})=&\sum_{i=1}^{N}J_{ei}(u_{e})\nonumber\\
	=&\mathbb{E} \bigg[ \int_{0}^{\infty}\sum_{i=1}^{N}\big(x_{ei}^{\top}Q_ex_{ei}-(x_{e}^{(N)})^{\top}Q_{\Gamma_e}x_{e}^{(N)}\nonumber\\
	&+u_{ei}^{\top}R_eu_{ei}\big)d\tau\bigg],
\end{align}
where $Q_{\Gamma_e}=\Gamma_e^{\top}Q_e+Q_e\Gamma_e-\Gamma_e^{\top}Q_e\Gamma_e\in \mathbb{S}^{n}$.

\begin{defn}\label{d1}
	The system \eqref{x_e} is said to be mean-square stabilizable if there exists a constant matrix $F\in \mathbb{R}^{m\times n}$ such that the (unique) solution to 
	\[dx_{ei}(t) = (A-BF)x_{ei}(t)dt + (C-DF)x_{ei}(t)dw_i(t)\] 
	satisfies $\lim_{t\to \infty}\mathbb{E}[\|x_{ei}(t)\|^2]=0$ for any $x_{ei}^0$. In this case, $F$ is called a stabilizer of system \eqref{x_e}, and the feedback control $u_{ei}(\cdot) = Fx_{ei}(\cdot)$ is called stabilizing.
\end{defn}
We now present several equivalent conditions for verifying the stabilizability of system \eqref{x_e}, which play an important role in the subsequent analysis.

\begin{lem}[\cite{Rami2000}]\label{le:stable}
	The three statements below are equivalent.
	\begin{enumerate}
		\item[1)] System \eqref{x_e} is mean-square stabilizable.
		\item[2)] There exists a matrix $F$ and a symmetric positive definite matrix $X>0$ such that
		\begin{align*}
			&(A-BF)^{\top}X+X(A-BF)\cr
			&+(C-DF)^{\top}X(C-DF)< 0.
		\end{align*}
		\item[3)] There exists a matrix $F$ such that for any given symmetric matrix $Y$, the linear matrix equation
		\begin{align*}
			&(A-BF)^{\top}X+X(A-BF)\cr
			&+(C-DF)^{\top}X(C-DF)+Y = 0
		\end{align*}
		admits a unique solution $X$. Moreover, if $Y > 0$ (respectively, $Y\ge 0$), then $X > 0$ (respectively, $X\ge 0$). 
	\end{enumerate}
\end{lem}
\begin{rem}
	When $C=0$ and $D=0$, the conditions in Lemma~\ref{le:stable} reduce to the following forms:
	for condition 2), $(A-BF)^{\top}X+X(A-BF)< 0$;
	for condition 3), $(A-BF)^{\top}X+X(A-BF)+Y = 0$.
	These reduced forms coincide with the stabilizability condition for the deterministic system $\dot{x}_{ei}=(A-BF)x_{ei}$.
\end{rem}

\begin{assum}\label{a:MFS}
	System \eqref{x_e} is mean-square stabilizable.
\end{assum}

Under Assumption~\ref{a:MFS}, the set of admissible control policies for each expert agent $i\in \mathcal{N}$ is defined as
\begin{align*}
	\mathcal{U}_{ad}^i=\bigg\{&u_{ei} \mid u_{ei}(t)\ \text{is adapted to}\ \sigma(x_{ei}(s), 0\le s \le t),\\ &\mathbb{E}\int_{0}^{\infty}\|x_{ei}(t)\|^2dt < +\infty \bigg\}.
\end{align*}
Based on \cite{Huang2012} and \cite{Wang2021}, the objective of the LQ MF social control problem is to find a set of admissible policies $\hat{u}_e=\{\hat{u}_{e1}, \dots , \hat{u}_{eN}\}$ that is asymptotically socially optimal, such that
\begin{align*}
	\bigg|\frac{1}{N}J_e^{soc}(\hat{u}_e)-\frac{1}{N}\inf_{u_{ei}\in \mathcal{U}_{ad}^i, i\in \mathcal{N}}J_e^{soc}({u}_e)\bigg|=O\bigg(\frac{1}{\sqrt{N}}\bigg).
\end{align*}
The corresponding SAREs for the expert population are given by
\begin{align}
	&A^{\top}P_{e}+P_{e}A+C^{\top}P_{e}C-(B^{\top}P_{e}+D^{\top}P_{e}C)^{\top}\cr
	&\times \Lambda_e^{-1}(B^{\top}P_{e}+D^{\top}P_{e}C)+Q_e=0\label{P_e},\\
	&A^{\top}\Pi_{e}+\Pi_{e}A+C^{\top}P_{e}C-(B^{\top}\Pi_{e}+D^{\top}P_{e}C)^{\top}\cr
	&\times \Lambda_e^{-1}(B^{\top}\Pi_{e}+D^{\top}P_{e}C)+Q_e-Q_{\Gamma_e}=0\label{Pi_e},
\end{align}
where $\Theta_e=D^{\top}P_{e}D$ and $\Lambda_e =R_e+\Theta_e$. Here, $P_{e}\in \mathbb{S}^n$ and $\Pi_{e}\in \mathbb{S}^n$ are the solutions to the two SAREs, respectively. From \cite{Wang2021}, the social optimal control policies, taken as the control inputs of the expert population, are given by
\begin{align}\label{u_e}
	u_{ei} \triangleq & -K_{e}x_{ei}-\bar{K}_{e}\bar{x}_{e}\cr
	=&- \Lambda_e^{-1}(B^{\top}P_{e}+D^{\top}P_{e}C)x_{ei}\cr
	&-\Lambda_e^{-1}B^{\top}(\Pi_{e}-P_{e})\bar{x}_e,\quad \forall i\in \mathcal{N},
\end{align}
where $K_e\triangleq \Lambda_e^{-1}(B^{\top}P_{e}+D^{\top}P_{e}C)$ and $\bar{K}_{e}\triangleq \Lambda_e^{-1}B^{\top}(\Pi_{e}-P_{e})$. Here, $\bar{x}_{e}=\mathbb{E}[x_{ei}]$ is the MF state, which serves as an approximation to $x_{e}^{(N)}$ and satisfies
\begin{align}\label{barx_e}
	\dot{\bar{x}}_e(t)=\big(A-B(K_e+\bar{K}_{e})\big)\bar{x}_e(t), \ \bar{x}_e(0)=\bar{x}_e^0.
\end{align} 
Under Assumption~\ref{a:MFS}, the social cost function \eqref{J_e} has an upper bound, that is, $J_e^{soc}<+\infty$. However, since $Q_e$, $R_e$ and $Q_{\Gamma_e}$ in \eqref{J_e} are indefinite, \eqref{J_e} may be unbounded below, making the MF social control problem ill-posed. Thus, we adopt a linear matrix inequality approach, which can simultaneously handle the existence of solutions to the SAREs and the well-posedness of the social cost function through convex optimization.  

The solvability of SAREs \eqref{P_e} and \eqref{Pi_e} is equivalent to the existence of $P_e, \Pi_e \in \mathbb{S}^n$ satisfying
\begin{align*}
	\mathcal{M}(P_e)&\triangleq \begin{pmatrix}
		\mathcal{A}_{P_e}+Q_e & P_e B + C^\top P_e D \\
		B^\top P_e + D^\top P_e C & R_e + D^\top P_e D
	\end{pmatrix} \ge 0,\\
	\mathcal{M}(\Pi_e)&\triangleq \begin{pmatrix}
		\mathcal{A}_{\Pi_e}+Q_e-Q_{\Gamma_e} & \Pi_e B + C^\top P_e D \\
		B^\top \Pi_e + D^\top P_e C & R_e + D^\top P_e D
	\end{pmatrix} \ge 0,
\end{align*}
where $\mathcal{A}_{P_e}=A^\top P_e + P_e A + C^\top P_e C$, $\mathcal{A}_{\Pi_e}=A^\top \Pi_e + \Pi_e A + C^\top P_e C$, and $R_e + D^\top P_e D > 0$. Let us define the following set $\mathcal{H}$ as 
\begin{align*}
	\mathcal{H}(P_e,\Pi_e,R_e)\triangleq \{&(P_e,\Pi_e)=(P_e^{\top},\Pi_e^{\top})\mid \mathcal{M}(P_e)\ge 0,\cr &\mathcal{M}(\Pi_e)\ge 0, R_e + D^\top P_e D > 0\}.
\end{align*}

\begin{assum}\label{a:lmi}
	$\mathcal{H}(P_e,\Pi_e,R_e) \ne \varnothing$, and $\mathcal{H}(P_e,\Pi_e,R_e)$ has a nonempty interior $(\tilde{P}_e,\tilde{\Pi}_e)$, in the sense that $\mathcal{M}(P_e)> 0$ and $\mathcal{M}(\Pi_e)> 0$.
\end{assum}

\begin{lem}[\cite{Rami2001}]\label{stabilizing}
	Subject to Assumptions~\ref{a:MFS} and \ref{a:lmi}, SAREs \eqref{P_e} and \eqref{Pi_e} admit a stabilizing solution, respectively. Moreover, the corresponding gains $K_e$ and $\bar{K}_e$ ensure that $A-BK_e$ and $A-B(K_e+\bar{K}_e)$ are Hurwitz.
\end{lem}
For the convenience of subsequent analysis, we define $S_e\triangleq \Pi_e-P_e$. Subtracting \eqref{P_e} from \eqref{Pi_e} yields
\begin{align}\label{S_e}
	&(A-BK_e)^{\top}S_e+S_e(A-BK_e)\cr
	&-S_eB\Lambda_e^{-1}B^{\top}S_e-Q_{\Gamma_e}=0.
\end{align}
\begin{rem}
	Compared with \eqref{Pi_e}, the terms $C^{\top}P_{e}C$ and $D^{\top}P_{e}C$ have disappeared in \eqref{S_e}, such that solving for $S_e$ does not directly depend on the unknown system matrices $C$ and $D$.
\end{rem}

\subsection{Inverse RL Problem for Learner Population}\label{1.2}
Consider a learner population with $N$ agents. The trajectory of the $i$th learner agent is governed by 
\begin{align}\label{x}
	dx_{i}(t) =& [Ax_{i}(t)+Bu_{i}(t)]dt \cr
	&+ [Cx_{i}(t)+Du_{i}(t)]dw_i(t), \ t\ge 0, \ i\in\mathcal{N},
\end{align}
where $x_{i}(t)\in\mathbb{R}^n$ and $u_{i}(t)\in\mathbb{R}^{m}$ are the state and control input of the learner agent, respectively. 
Denote the set of states and control inputs for the learner population as $x = \{x_{1}, \dots , x_{N}\}$ and $u = \{u_{1}, \dots , u_{N}\}$, respectively. Define the cost function for each learner agent as 
\begin{align}\label{Ji}
	J_{i}(u)=\mathbb{E} \left[ \int_{0}^{\infty}(\|x_{i}-\Gamma x^{(N)}\|_{Q}^2+\|u_{i}\|_{R}^2) d\tau\right], 
\end{align}
where $Q=Q^{\top}\in \mathbb{S}^{n}$, $R=R^{\top}\in \mathbb{S}^{m}$, $\Gamma\in \mathbb{R}^{n\times n}$, and $x^{(N)}=\frac{1}{N}\sum_{j=1}^{N}x_{j}$. Similarly, $Q$ and $R$ are not required to be positive (semi-)definite. The social cost function for the learner population is given by
\begin{align}\label{Jsoc}
	J^{soc}(u)
	=&\mathbb{E} \bigg[ \int_{0}^{\infty}\sum_{i=1}^{N}\big(x_{i}^{\top}Qx_{i}-(x^{(N)})^{\top}Q_{\Gamma}x^{(N)}\cr
	&+u_{i}^{\top}Ru_{i}\big)d\tau\bigg],
\end{align}
where $Q_{\Gamma}=\Gamma^{\top}Q+Q\Gamma-\Gamma^{\top}Q\Gamma\in \mathbb{S}^{n}$.
Define $\Theta=D^{\top}PD$ and $\Lambda =R+\Theta$. The social optimal control policy $u_{i}$ of each learner agent $i\in \mathcal{N}$ is given by
\begin{align}\label{u}
	u_{i} \triangleq&  -Kx_{i}-\bar{K}\bar{x}\cr
	=&- \Lambda^{-1}(B^{\top}P+D^{\top}PC)x_{i}\cr
	&-\Lambda^{-1}B^{\top}(\Pi-P)\bar{x}, 
\end{align}
where $K\triangleq \Lambda^{-1}(B^{\top}P+D^{\top}PC)$ and $\bar{K}\triangleq \Lambda^{-1}B^{\top}(\Pi-P)$. In addition, $P\in \mathbb{S}^{n}$ and $\Pi\in \mathbb{S}^{n}$ are the solutions to the following SAREs
\begin{align}
	&A^{\top}P+PA+C^{\top}PC-(B^{\top}P+D^{\top}PC)^{\top}\cr
	&\times \Lambda^{-1}(B^{\top}P+D^{\top}PC)+Q=0\label{P},\\
	&A^{\top}\Pi+\Pi A+C^{\top}PC-(B^{\top}\Pi+D^{\top}PC)^{\top}\cr
	&\times \Lambda^{-1}(B^{\top}\Pi+D^{\top}PC)+Q-Q_{\Gamma}=0\label{Pi},
\end{align}respectively. Let $\bar{x}=\mathbb{E}[x_{i}]$ be the MF state, which is an approximation to $x^{(N)}$ satisfying
\begin{align}\label{barx}
	\dot{\bar{x}}(t)=\big(A-B(K+\bar{K})\big)\bar{x}(t), \ \bar{x}(0)=\bar{x}^0.
\end{align}
Similarly, the learner population also satisfies Assumptions~\ref{a:initial_state}--\ref{a:lmi}, which will be directly used for the learner in the following analysis.
Define $S\triangleq \Pi-P$, then combine \eqref{P} and \eqref{Pi} to obtain 
\begin{align}\label{S}
	(A-BK)^{\top}S+S(A-BK)-SB\Lambda^{-1}B^{\top}S-Q_{\Gamma}=0.
\end{align}

\begin{rem}
	Similar to the expert case, the derived equation \eqref{S} is independent of $C$ and $D$. For brevity, we employ \eqref{S} rather than \eqref{Pi} throughout the subsequent analysis. This treatment is also consistent with the dynamic characterization of the MF state.
\end{rem}

We next impose a necessary assumption and a related definition to facilitate the subsequent inverse RL theoretical analysis.
\begin{assum}\label{a1}
	The weights $Q_e$, $R_e$ and $Q_{\Gamma_e}$ in the expert's social cost function \eqref{J_e}, the gains $K_e$ and $\bar{K}_e$ in \eqref{u_e}, as well as the solutions $P_e$ and $\Pi_{e}$ to the SAREs \eqref{P_e} and \eqref{Pi_e} are unknown to the learner. The learner can only obtain the state dataset $x_e$ and control dataset $u_e$ by observing the trajectories of the expert population.
\end{assum}

\begin{defn}
	Give a constant matrix $R$ with an appropriate dimension. If the indefinite weights $Q$ and $Q_{\Gamma}$ in SAREs \eqref{P} and \eqref{Pi} yield $K=K_{e}$ and $\bar{K}=\bar{K}_{e}$ for all agents, then $Q$ and $Q_{\Gamma}$ are called equivalent weights to $Q_e$ and $Q_{\Gamma_e}$ in \eqref{P_e} and \eqref{Pi_e}, respectively.
\end{defn}

In this article, we mainly study the following inverse RL problem.
\begin{prob}\label{IRL}
	Let Assumptions~\ref{a:initial_state}--\ref{a1} hold. Using only observed trajectory data, the learner population aims to reconstruct the unknown social cost function for imitating the expert’s trajectories. This is achieved by finding the equivalent weights $Q$ and $Q_\Gamma$ to obtain $K = K_e$ and $\bar{K}=\bar{K}_e$ for any $R\in \mathbb{R}^{m\times m}$. Furthermore, the control gains $K$ and $\bar{K}$ obtained in each iterative process are desired to stabilize the learner agents \eqref{x}. Since the symmetric matrices $Q$, $R$ and $Q_{\Gamma}$ are indefinite, this is called an inverse RL problem for indefinite LQ MF social optimization.
\end{prob}

\section{Model-Based Framework for the Inverse RL Problem}\label{2}
In this section, we address the inverse RL problem using a model-based framework. Utilizing system dynamics and observed trajectory data from the expert, we construct an iterative learning scheme to obtain equivalent indefinite cost weights and recover the expert’s optimal control gains. The approach solves iterative Riccati and Lyapunov equations with explicit system information, which lays a rigorous theoretical foundation for the subsequent model-free development.
\subsection{Model-Based Inverse RL via Policy Iteration}\label{2.1}
First, we present a theorem that provides new Lyapunov functions for solving the Inverse RL problem based on model information. 
\begin{thm}\label{t1}
	Let Assumption~\ref{a:lmi} hold. If the weights $Q$, $Q_\Gamma$, $R$ and the solutions $P$, $\Pi$ satisfy the two SAREs \eqref{P} and \eqref{Pi} for the learner and the following Lyapunov equations
	\begin{align}
		&(A-BK_e)^{\top}P+P(A-BK_e)+Q+K_e^{\top}RK_e\nonumber\\
		&+(C-DK_e)^{\top}P(C-DK_e)=0\label{Pk},\\
		&\big(A-B(K+\bar{K}_e)\big)^{\top}S+S\big(A-B(K+\bar{K}_e)\big)\nonumber\\
		&+\bar{K}_e^{\top}\Lambda\bar{K}_e-Q_{\Gamma}=0\label{Sk},
	\end{align}
	then the gain $K$ is equal to $K_e$, and $\bar{K}$ is equal to $\bar{K}_e$. 
\end{thm}
\begin{pf}
	Since $K=\Lambda^{-1}(B^{\top}P+D^{\top}PC)$ from \eqref{u}, equation \eqref{P} can be transformed into
	\begin{align*}
		A^{\top}P+PA+C^{\top}PC-K^{\top}\Lambda K+Q=0.
	\end{align*}
	Introducing $K_e$ from \eqref{u_e} into the above equation gives 
	\begin{align*}
		&(A-BK_e)^{\top}P+P(A-BK_e)+Q\nonumber\\
		&+(C-DK_e)^{\top}P(C-DK_e)+K_e^{\top}\Lambda K+K^{\top}\Lambda K_e\nonumber\\
		&-K^{\top}\Lambda K-K_e^{\top}\Theta K_e=0.
	\end{align*}
	Subtracting the above equation from \eqref{Pk} yields
	\begin{align}\label{K_e-K}
		(K_e-K)^{\top}\Lambda(K_e-K)=0.
	\end{align}
	By Assumption~\ref{a:lmi} that $\Lambda=R+D^{\top}PD>0$, equation \eqref{K_e-K} implies that $K=K_e$.
	
	Similarly, substituting $\bar{K}$ and $\bar{K}_e$ into \eqref{S} yields 
	\begin{align*}
		&\big(A-B(K+\bar{K}_e)\big)^{\top}S+S\big(A-B(K+\bar{K}_e)\big)\nonumber\\
		&+\bar{K}_e^{\top}\Lambda\bar{K}+\bar{K}^{\top}\Lambda \bar{K}_e-\bar{K}^{\top}\Lambda\bar{K}-Q_{\Gamma}=0.
	\end{align*}		
	Subtracting this equation from \eqref{Sk} derives
	\begin{align}\label{K_e-barK}
		(\bar{K}_e-\bar{K})^{\top}\Lambda(\bar{K}_e-\bar{K})=0.
	\end{align}
	Since $\Lambda$ is positive definite, it follows from \eqref{K_e-barK} that $\bar{K}=\bar{K}_e$. This completes the proof.   \hfill $\square$
\end{pf}

Then, by collecting the trajectory data $x_{ei}$ and $u_{ei}$ from expert agents, we can recover the control gains $K_e$ and $\bar{K}_e$ in \eqref{u_e}. The control law \eqref{u_e} for expert agent $i$ can be rearranged as
\[
	U_{ei} = -\tilde{K}_eX_{e},
\]
where $U_{ei}=[u_{ei}(t-(k-1)T),\dots,u_{ei}(t-T),u_{ei}(t)]$, $\tilde{K}_e=[K_e, \bar{K}_e]$, $X_{e}=[X_{ei}, \bar{X}_e]^{\top}$, $X_{ei}=[x_{ei}(t-(k-1)T),\dots,x_{ei}(t-T),x_{ei}(t)]^{\top}$, and $\bar{X}_e=[\bar{x}_{e}(t-(k-1)T),\dots,\bar{x}_{e}(t-T),\bar{x}_{e}(t)]^{\top}$.
Using the batch least-square method, we obtain the estimation $\hat{\tilde{K}}_e$ of $\tilde{K}_e$ as
\begin{equation}\label{K_e}
	\hat{\tilde{K}}_e=-U_{ei}X_e^{\top}(X_eX_e^{\top})^{-1}.
\end{equation}
With sufficient trajectory data collected, the least-square estimate converges to the true gain, i.e., $\hat{\tilde{K}}_e=\tilde{K}_e$. Then Algorithm \ref{algorithm1} is proposed for model-based inverse RL implementation based on Theorem~\ref{t1}.

\begin{algorithm}[!t]
	\caption{A model-based inverse RL algorithm for indefinite MF social optimization.}\label{algorithm1}
	\begin{algorithmic}[1]
		\STATE \textbf{Initialization:} Give an arbitrary matrix $R$, and 
		select appropriate initial weight matrices $Q^0$ and $\Gamma^0$. Calculate $Q_{\Gamma}^0$ by $Q_{\Gamma}^0=(\Gamma^0)^{\top}Q^0+Q^0\Gamma^0-(\Gamma^0)^{\top}Q^0\Gamma^0$, and $\tilde{K}_e$ by \eqref{K_e}.
		\STATE \textbf{Loop 1:} Set $l=0$, and $\epsilon_q$ as a small threshold.
		\REPEAT
		\STATE \textbf{Policy evaluation:} Calculate $P^l$ by the following Lyapunov equation
		\vspace{-0.8em}
		\begin{align}\label{Pl}
			&(A-BK_e)^{\top}P^l+P^l(A-BK_e)+(C-DK_e)^{\top}\nonumber\\
			&\times P^l(C-DK_e)=-Q^{l}-K_e^{\top}RK_e.
		\end{align}
	    \vspace{-1.5em}
		\STATE \textbf{Policy improvement:} Solve for $K^{l}$ by
		\vspace{-0.5em}
		\begin{align}\label{Kl}
			K^{l}=(\Lambda^{l})^{-1}(B^{\top}P^l+D^{\top}P^lC),
		\end{align}
		where $\Lambda^{l}\triangleq R+D^{\top}P^lD$.
		\STATE \textbf{Weight improvement:} Update $Q^{l+1}$ by
		\vspace{-0.5em}
		\begin{align}\label{Ql}
			Q^{l+1}=-A^{\top}P^l-P^lA-C^{\top}P^lC+(K^{l})^{\top}\Lambda^{l}K^{l}.
		\end{align}
	    \vspace{-1.5em}
		\STATE Let $l=l+1$.
		\UNTIL{$\|Q^{l}-Q^{l-1}\|\le \epsilon_q$.}
		\STATE Let the final values $K^* = K^{l-1}$, $P^* = P^{l-1}$, $Q^* = Q^{l}$, and $\Lambda^* = R+D^{\top}P^*D$.
		
		\STATE \textbf{Loop 2:} Set $k=0$, and $\epsilon_{\gamma}$ as a small threshold. 
		\REPEAT
		\STATE \textbf{MF policy evaluation:} Calculate $S^{k}$ by the following Lyapunov equation
		\vspace{-0.8em}
		\begin{align}\label{Sl}
			&\big(A-B(K^*+\bar{K}_e)\big)^{\top}S^k+S^k\big(A-B(K^*+\bar{K}_e)\big)\nonumber\\
			&=-\bar{K}_e^{\top}\Lambda^*\bar{K}_e+Q_{\Gamma}^k.
		\end{align}
	    \vspace{-1.5em}
		\STATE \textbf{MF policy improvement:} Solve for $\bar{K}^{k}$ by
		\vspace{-0.5em}
		\begin{align}\label{barKl}
			\bar{K}^{k}= (\Lambda^*)^{-1}B^{\top}S^{k}.
		\end{align}
	    \vspace{-1.5em}
		\STATE \textbf{MF weight improvement:} Update $Q_{\Gamma}^{k+1}$ by
		\vspace{-0.5em}
		\begin{align}\label{Gammal}
			Q_{\Gamma}^{k+1}=&(A-BK^*)^{\top}S^{k}+S^{k}(A-BK^*)\cr
			&-(\bar{K}^{k})^{\top}\Lambda^*\bar{K}^{k}.
		\end{align} 
	    \vspace{-1.5em}
		\STATE Let $k=k+1$.
		\UNTIL{$\|Q_{\Gamma}^{k}-Q_{\Gamma}^{k-1}\|\le \epsilon_{\gamma}$.}
		\STATE Let final values $\bar{K}^* = \bar{K}^{k-1}$, $S^* = S^{k-1}$, and $Q_{\Gamma}^* = Q_{\Gamma}^{k}$.
		Moreover, $\Pi^*=P^*+S^*$, and the final policy for each learner agent $i$ is given by $u_{i}= -K^*x_{i}-\bar{K}^*\bar{x},\ i\in \mathcal{N}$.	
	\end{algorithmic}
\end{algorithm}
\begin{rem}
	Unlike standard RL, which optimizes policies for a given cost function, inverse RL recovers unknown indefinite cost weights from expert trajectories to replicate optimal behaviors. The proposed algorithm contains two sequential iterative loops, which separately handle individual and MF dynamics. Each iteration in Loops~1 and 2 consists of three steps: the first two steps follow the classic PI structure, while the third step performs additional indefinite weight updating. This constitutes the key distinction from conventional RL-based PI for MF social optimization in \cite{Xu2025}.
\end{rem}
\begin{rem}
	Although the cost weights are allowed to be unknown and indefinite, we fix $R$ and iterate only $Q$ and $Q_{\Gamma}$ in Algorithm~\ref{algorithm1}. On the one hand, fixing $R$ provides proper degrees of freedom for solving the related Lyapunov equations and SAREs, and ensures that $\Lambda>0$, which is needed for stability and the existence of solutions. On the other hand, this setting avoids trivial solutions in the learning process \citep{Johnson2013}.
\end{rem}
\begin{rem}
	By appropriately selecting the initial weight matrices $Q^0$ and $\Gamma^0$, the resulting initial gains $K^0$ and $\bar{K}^0$ can be guaranteed to be stabilizing. This initialization setting also ensures the positive definiteness of $\Lambda^0$, which is critical for the convergence and closed-loop stability of the subsequent iterations.
\end{rem}

\subsection{Analysis of Algorithm \ref{algorithm1}}\label{2.2}
In what follows, we prove the closed-loop stability under the learned policies and analyze the convergence of Algorithm~\ref{algorithm1}. Besides, it will be shown that $Q$, $Q_{\Gamma}$, $P$ and $S$ learned by Algorithm \ref{algorithm1} for $K_e$ and $\bar{K}_e$ may not be unique.

Based on Algorithm~\ref{algorithm1}, the final values $P^*$, $Q^*$, $K^*$, $S^*$, $Q_{\Gamma}^*$ and $\bar{K}^*$ are  solutions satisfying Theorem \ref{t1}, that is,
\begin{align}
	&A^{\top}P^*+P^*A+C^{\top}P^*C-(B^{\top}P^*+D^{\top}P^*C)^{\top}\nonumber\\
	&\times (\Lambda^*)^{-1}(B^{\top}P^*+D^{\top}P^*C)+Q^*=0,\label{P*}\\
	&(A-BK^*)^{\top}S^*+S^*(A-BK^*)-S^*B\nonumber\\
	&\times (\Lambda^*)^{-1}B^{\top}S^*-Q_{\Gamma}^*=0,\label{S*}\\
	&(A-BK_e)^{\top}P^*+P^*(A-BK_e)+Q^*+K_e^{\top}RK_e\nonumber\\
	&+(C-DK_e)^{\top}P^*(C-DK_e)=0\label{P*k},\\
	&\big(A-B(K^*+\bar{K}_e)\big)^{\top}S^{*}+S^{*}\big(A-B(K^*+\bar{K}_e)\big)\nonumber\\
	&-Q_{\Gamma}^{*}+(\bar{K}_e)^{\top}\Lambda^*\bar{K}_{e}=0,\label{S*k}\\
	&K^*=(\Lambda^*)^{-1}(B^{\top}P^*+D^{\top}P^*C),\label{K*}\\
	&\bar{K}^*=(\Lambda^*)^{-1}B^{\top}S^*,\label{barK*}
\end{align}
where $\Lambda^*\triangleq R+D^{\top}P^*D$. The following two theorems analyze the convergence of sequences $\{Q^{l},P^l,K^{l}\}_{l=0}^{\infty}$ and $\{Q_{\Gamma}^{k},S^k,\bar{K}^{k}\}_{k=0}^{\infty}$, respectively, and the stability of corresponding closed-loop systems.

\begin{thm}\label{convergence_P}
	Let Assumptions~\ref{a:initial_state}--\ref{a:lmi} hold. Give an arbitrary weight $R$, and an initial $Q^0$ that satisfies $Q^0\le Q^*$. In addition, $Q^0$ is chosen such that $\Lambda^0=R+D^{\top}P^0D>0$ and $K^0$ is a stabilizer. Then the sequence $\{Q^{l},P^l,K^{l}\}_{l=0}^{\infty}$ generated by Algorithm \ref{algorithm1} has the following properties for all $l\ge 0$. 
	\begin{enumerate}
		\item [(1)] $Q^{l}\le Q^{l+1}\le Q^*$, and $P^{l}\le P^{l+1}\le P^*$.
		\item [(2)] $A-BK^{l}$ is Hurwitz. 
		\item [(3)] $\lim_{l\to \infty} Q^l=Q^*$, $\lim_{l\to \infty} P^l=P^*$, and \\$\lim_{l\to \infty} K^l=K^*=K_e$.  
	\end{enumerate}  
\end{thm}
\begin{pf}
	First, subtracting \eqref{Pl} from \eqref{Ql} obtains
	\begin{align}\label{Ql+1-Ql}
		Q^{l+1}-Q^{l}=&K_e^{\top}\Lambda^lK_e+(K^{l})^{\top}\Lambda^lK^{l}\cr
		&-K_e^{\top}(B^{\top}P^l+D^{\top}P^lC)\cr
		&-(B^{\top}P^l+D^{\top}P^lC)^{\top}K_e.
	\end{align}
	Since $\Lambda^lK^{l}=B^{\top}P^l+D^{\top}P^lC$ from \eqref{Kl}, equation \eqref{Ql+1-Ql} can be transformed into 
	\begin{align}\label{Ql+1l}
		Q^{l+1}-Q^{l}=(K_e-K^{l})^{\top}\Lambda^l(K_e-K^{l}).
	\end{align}
	Then subtracting \eqref{Pl} at iteration $l$ from that at iteration $l+1$ gives
	\begin{align}\label{Pl+1Pl}
		&(A-BK_e)^{\top}(P^{l+1}-P^l)+(P^{l+1}-P^l)(A-BK_e)\nonumber\\
		&+(C-DK_e)^{\top}(P^{l+1}-P^l)(C-DK_e)\nonumber\\
		=&\ Q^{l}-Q^{l+1},
	\end{align}
	When $l=0$, we ensure that $\Lambda^0=R+D^{\top}P^0D>0$, such that $Q^0\le Q^1$ by \eqref{Ql+1l}. Since $K_e$ is the stabilizer of expert agents \eqref{x_e}, based on Lemma~\ref{le:stable}, \eqref{Pl+1Pl} implies that $P^0\le P^1$, which can deduce that $0<\Lambda^0\le \Lambda^1$. Thus, by forward recursion, $\Lambda^l=R+D^{\top}P^lD>0$ for all iterations, which satisfies Assumption~\ref{a:lmi}. From the above analysis, $\{Q^l\}_{l=0}^{\infty}$ and $\{P^l\}_{l=0}^{\infty}$ are increasing sequences. Based on \eqref{Ql+1l}, $Q^l= Q^{l+1}$ if and only if $K^l=K_e$. 
	
	Next, we will analyze the stability and prove the result (2) in two cases.\\
	\noindent {\bf(i) The standard case}, where $Q^0\ge 0$ and $R>0$. In fact, this case implies that $Q^{l}\ge 0$ for all $l\ge 0$. Then, since $K_e$ is the stabilizer of expert agents \eqref{x_e}, $P^l>0$ is the unique solution to Lyapunov recursion \eqref{Pl} by Lemma~\ref{le:stable}. With \eqref{Pl} and \eqref{Kl}, we notice that
	\begin{align}\label{Kle}
		&(A-BK^{l})^{\top}P^l+P^l(A-BK^{l})\cr
		&+(C-DK^{l})^{\top}P^l(C-DK^{l})\cr
		=&(A-BK_e)^{\top}P^l+P^l(A-BK_e)-K_e^{\top}D^{\top}P^lDK_e\cr
		&+(C-DK_e)^{\top}P^l(C-DK_e)+K_e^{\top}\Lambda^lK^{l}\cr
		&+(K^{l})^{\top}\Lambda^lK_e-(K^{l})^{\top}\Lambda^lK^{l}-(K^{l})^{\top}RK^{l}\cr
		=&-Q^{l}-(K_e-K^{l})^{\top}\Lambda^l(K_e-K^{l})-(K^{l})^{\top}RK^{l},
	\end{align}
	Since $Q^{l}\ge 0$, $\Lambda^l > 0$ and $R>0$, then $(A-BK^{l})^{\top}P^l+P^l(A-BK^{l})+(C-DK^{l})^{\top}P^l(C-DK^{l})<0$. Based on Lemma~\ref{le:stable}, $K^{l}$ is a stabilizer of learner agents for all $l\ge 0$.
	\\ {\bf(ii) The general case}, where $Q^0$ and $R$ are indefinite matrices. Similar to obtaining \eqref{Kle}, we have 
	\begin{align}\label{KlPl+1}
		&(A-BK^l)^{\top}P^{l+1}+P^{l+1}(A-BK^l)\nonumber\\
		&+(C-DK^l)^{\top}P^{l+1}(C-DK^l)\nonumber\\
		=&-Q^{l+1}-(K_e-K^{l+1})^{\top}\Lambda^{l+1}(K_e-K^{l+1})\nonumber\\
		&+(K^{l+1}-K^{l})^{\top}\Lambda^{l+1}(K^{l+1}-K^{l})-(K^l)^{\top}RK^l.
	\end{align}
	By \eqref{Kl}, the SARE \eqref{Ql} is reformulated as 
	\begin{align}\label{KlPl}
		&(A-BK^l)^{\top}P^l+P^l(A-BK^l)\nonumber\\
		&+(C-DK^l)^{\top}P^l(C-DK^l)\nonumber\\
		=&-Q^{l+1}-(K^l)^{\top}RK^l.
	\end{align}
	Then subtracting \eqref{KlPl} from \eqref{KlPl+1} yields
	\begin{align}\label{KlPl+1Pl}
		&(A-BK^l)^{\top}(P^{l+1}-P^l)+(P^{l+1}-P^l)(A-BK^l)\nonumber\\
		&+(C-DK^l)^{\top}(P^{l+1}-P^l)(C-DK^l)\cr
		=&-(K_e-K^{l+1})^{\top}\Lambda^{l+1}(K_e-K^{l+1})\nonumber\\
		&+(K^{l+1}-K^{l})^{\top}\Lambda^{l+1}(K^{l+1}-K^{l}).
	\end{align}
	Equation \eqref{KlPl+1Pl} can be transformed into 	
    \begin{align}\label{Kl+1Pl+1Pl}
		&(A\!-\!BK^{l+1})^{\top}\!(P^{l+1}\!-\!P^l)+(P^{l+1}\!-\!P^l)(A-BK^{l+1})\nonumber\\
		&+(C-DK^{l+1})^{\top}(P^{l+1}-P^l)(C-DK^{l+1})\nonumber\\
		=&-(K_e-K^{l+1})^{\top}\Lambda^{l+1}(K_e-K^{l+1})\nonumber\\
		&-(K^{l+1}-K^{l})^{\top}\Lambda^{l}(K^{l+1}-K^{l}).
	\end{align}
	Since $P^{l+1}-P^l\ge 0$ and $\Lambda^{l+1}\ge\Lambda^l> 0$, by Lemma~\ref{le:stable}, $K^{l+1}$ is a stabilizer for all $l\ge 0$. Note that the initial gain $K^0$ can be obtained as a stabilizer by choosing a suitable initial weight $Q^0$. Combining the above analysis, it follows that $A-BK^{l}$ is Hurwitz for all $l\ge 0$. 
	
	Then, we use mathematical induction to show that $Q^l$ and $P^l$ are bounded above. When $l=0$, note that $Q^0\le Q^*$. Suppose $Q^l\le Q^*$ for $l\ge 0$, we subtract \eqref{Pl} from \eqref{P*k} to obtain
	\begin{align}\label{P*l}
		&(A-BK_e)^{\top}(P^*-P^l)+(P^*-P^l)(A-BK_e)\nonumber\\
		&+(C-DK_e)^{\top}(P^*-P^l)(C-DK_e)\nonumber\\
		=&\ Q^l-Q^*\le 0.
	\end{align}
	Based on Lemma~\ref{le:stable}, $P^l\le P^*$ for all iterations since $K_e$ is the stabilizer of expert agents. By \eqref{Kl}, the SARE \eqref{P*} is reformulated as 
	\begin{align}\label{KlP*}
		&(A-BK^l)^{\top}P^*+P^*(A-BK^l)\nonumber\\
		&+(C-DK^l)^{\top}P^*(C-DK^l)\cr
		=&-Q^*\!-\!(K^l)^{\top}\!RK^l \!+\!(K^*-K^l)^{\top}\!\Lambda^*(K^*-K^l).
	\end{align}
	Then based on Lemma~\ref{le:stable} and result (2), subtracting \eqref{KlPl} from \eqref{KlP*} yields 
	\begin{align*}
		&(A-BK^l)^{\top}(P^*-P^l)+(P^*-P^l)(A-BK^l)\nonumber\\
		&+(C-DK^l)^{\top}(P^*-P^l)(C-DK^l)\cr
		=&Q^{l+1}-Q^*+(K^*-K^l)^{\top}\Lambda^*(K^*-K^l)\le 0
	\end{align*}
	for $l\ge 0$. Therefore, the fact that $Q^{l+1}\le Q^*$ holds. From the above analysis, result (1) is established. 
	
	Next, we show that result (3) also holds. By \eqref{Pl}, \eqref{Kl} and \eqref{Ql}, the final values $P^*$, $Q^*$ and $K^*$ satisfy \eqref{P*}, \eqref{P*k} and \eqref{K*}. Thus $\lim_{l\to \infty} Q^l=Q^*$, $\lim_{l\to \infty} P^l=P^*$, and $\lim_{l\to \infty} K^l=K^*$. At the $(l+1)$th iteration, substituting \eqref{Ql} into \eqref{Pl}, we get
	\begin{align*}
		&A^{\top}P^{l+1}+P^{l+1}A+C^{\top}P^{l+1}C-K_e^{\top}(B^{\top}P^{l+1}\nonumber\\
		&+D^{\top}P^{l+1}C)-(B^{\top}P^{l+1}+D^{\top}P^{l+1}C)^{\top}K_e\nonumber\\
		&+K_e^{\top}(R+D^{\top}P^{l+1}D)K_e\nonumber\\
		=&\ A^{\top}P^{l}+P^{l}A+C^{\top}P^{l}C-(K^l)^{\top}(R+D^{\top}P^{l}D)K^l.
	\end{align*}
	When $\{P^l\}_{l=0}^{\infty}$ converges, $P^l=P^{l+1}\to P^*$, then the above equation becomes
	\begin{align*}
		(K_e-K^*)^{\top}\Lambda^*(K_e-K^*)=0.
	\end{align*}
	Since $\Lambda^*>0$, it can be deduced that $K^*=K_e$. This completes the proof. \hfill $\square$
\end{pf}

\begin{thm}\label{convergence_S}
	Adopt the setting and results of Theorem~\ref{convergence_P}. In addition, we select an initial weight $\Gamma^0$ such that $Q_{\Gamma}^0\ge Q_{\Gamma}^*$, and $A-B(K^*+\bar{K}^{0})$ is Hurwitz. Then the sequence $\{Q_{\Gamma}^{k},S^k,\bar{K}^{k}\}_{k=0}^{\infty}$ generated by Algorithm \ref{algorithm1} has the following properties for all $k\ge 0$.
	\begin{enumerate}
		\item [(1)] $Q_{\Gamma}^{k}\ge Q_{\Gamma}^{k+1}\ge Q_{\Gamma}^*$, and $S^{k}\le S^{k+1}\le S^*$.
		\item [(2)] $A-B(K^*+\bar{K}^{k})$ is Hurwitz. 
		\item [(3)] $\lim_{k\to \infty} Q_{\Gamma}^k=Q_{\Gamma}^*$, $\lim_{k\to \infty} S^k=S^*$, and \\ $\lim_{k\to \infty} \bar{K}^k=\bar{K}^*=\bar{K}_e$.
	\end{enumerate}  
\end{thm}
\begin{pf}
	First, subtracting \eqref{Sl} from \eqref{Gammal}, and substituting $\Lambda^*\bar{K}^{k}=B^{\top}S^k$ from \eqref{barKl} into the resulting difference yield
	\begin{align}\label{Gammal+1l}
		Q_{\Gamma}^{k+1}-Q_{\Gamma}^{k}=-(\bar{K}_e-\bar{K}^{k})^{\top}\Lambda^*(\bar{K}_e-\bar{K}^{k}).
	\end{align}
	From the proof of Theorem~\ref{convergence_P}, $\Lambda^*>0$, which implies that $Q_{\Gamma}^{k}\ge Q_{\Gamma}^{k+1}$, and $Q_{\Gamma}^{k}=Q_{\Gamma}^{k+1}$ if and only if $\bar{K}^k=\bar{K}_e$ based on \eqref{Gammal+1l}. Then subtract \eqref{Sl} at iteration $k$ from that at iteration $k+1$ to obtain
	\begin{align}\label{Pil+1l}
		&\big(A-B(K^*+\bar{K}_e)\big)^{\top}(S^{k+1}-S^k)\nonumber\\
		&+(S^{k+1}-S^k)\big(A-B(K^*+\bar{K}_e)\big)\nonumber\\
		=&\ Q_{\Gamma}^{k+1}-Q_{\Gamma}^{k}\le 0.
	\end{align}
	Note that $K^*=K_e$ from Theorem~\ref{convergence_P}, and $A-B(K_e+\bar{K}_e)$ is Hurwitz from Lemma~\ref{stabilizing}. Thus based on Lemma~\ref{le:stable}, $S^{k+1}\ge S^k$ from \eqref{Pil+1l}.
	
	Next, we will show that $A-B(K^*+\bar{K}^{k})$ is Hurwitz for all $k\ge 0$. According to \eqref{Sl} for iteration $k+1$, one has 
	\begin{align}\label{KlPil+1}
		&\big(A-B(K^*+\bar{K}^k)\big)^{\top}S^{k+1}+S^{k+1}\big(A-B(K^*+\bar{K}^k)\big)\nonumber\\
		=&\big(A-B(K^*+\bar{K}_e)\big)^{\top}S^{k+1}+S^{k+1}\big(A-B(K^*+\bar{K}_e)\big)\nonumber\\
		&+\bar{K}_e^{\top}\Lambda^*\bar{K}^{k+1}+(\bar{K}^{k+1})^{\top}\Lambda^*\bar{K}_e-(\bar{K}^{k})^{\top}\Lambda^*\bar{K}^{k+1}\nonumber\\
		&-(\bar{K}^{k+1})^{\top}\Lambda^*\bar{K}^{k}\nonumber\\
		=&\ Q_{\Gamma}^{k+1}-(\bar{K}_e-\bar{K}^{k+1})^{\top}\Lambda^*(\bar{K}_e-\bar{K}^{k+1})\nonumber\\
		&+(\bar{K}^{k+1}-\bar{K}^{k})^{\top}\Lambda^*(\bar{K}^{k+1}-\bar{K}^{k})-(\bar{K}^{k})^{\top}\Lambda^*\bar{K}^{k}.
	\end{align}
	By \eqref{barKl}, equation \eqref{Gammal} can be transformed into
	\begin{align}\label{KlPil}
		&\big(A-B(K^*+\bar{K}^k)\big)^{\top}S^{k}+S^{k}\big(A-B(K^*+\bar{K}^k)\big)\nonumber\\
		=&\ Q_{\Gamma}^{k+1}-(\bar{K}^{k})^{\top}\Lambda^*\bar{K}^{k}.
	\end{align}
	Then subtracting \eqref{KlPil} from \eqref{KlPil+1} yields
	\begin{align}\label{KlPil+1Pil}
		&\big(A-B(K^*+\bar{K}^k)\big)^{\top}(S^{k+1}-S^k)\nonumber\\
		&+(S^{k+1}-S^k)\big(A-B(K^*+\bar{K}^k)\big)\cr
		=&-(\bar{K}_e-\bar{K}^{k+1})^{\top}\Lambda^*(\bar{K}_e-\bar{K}^{k+1})\nonumber\\
		&+(\bar{K}^{k+1}-\bar{K}^{k})^{\top}\Lambda^*(\bar{K}^{k+1}-\bar{K}^{k}).
	\end{align}
	Using \eqref{KlPil+1Pil} leads to
	\begin{align}\label{Kl+1Pil+1Pil}
		&\big(A-B(K^*+\bar{K}^{k+1})\big)^{\top}(S^{k+1}-S^k)\nonumber\\
		&+(S^{k+1}-S^k)\big(A-B(K^*+\bar{K}^{k+1})\big)\cr
		=&-(\bar{K}_e-\bar{K}^{k+1})^{\top}\Lambda^*(\bar{K}_e-\bar{K}^{k+1})\nonumber\\
		&-(\bar{K}^{k+1}-\bar{K}^{k})^{\top}\Lambda^*(\bar{K}^{k+1}-\bar{K}^{k}).
	\end{align}
	Since $S^{k+1}-S^k\ge 0$ and $\Lambda^*> 0$, based on Lemma~\ref{le:stable}, $\bar{K}^{k+1}$ are the stabilizers for all $k\ge0$. Provided that an initial weight $\Gamma^0$ is chosen to make $A-B(K^*+\bar{K}^{0})$ Hurwitz, result (2) holds.
	
	We next prove the boundedness of the sequences $\{Q_{\Gamma}^{k}\}_{k=0}^{\infty}$ and $\{S^k\}_{k=0}^{\infty}$ in result (1) by mathematical induction. Select the initial weight $Q_{\Gamma}^0\ge Q_{\Gamma}^*$. Suppose that $Q_{\Gamma}^k \ge Q_{\Gamma}^*$ for $k\ge 0$, then subtract \eqref{Sl} from \eqref{S*k} to obtain
	\begin{align}\label{S*Sk}
		&\big(A-B(K^*+\bar{K}_e)\big)^{\top}(S^{*}-S^k)+(S^{*}-S^k)\big(A\nonumber\\
		&-B(K^*+\bar{K}_e)\big)=Q_{\Gamma}^{*}-Q_{\Gamma}^{k}\le 0,
	\end{align}
	which has the same form as \eqref{Pil+1l}. Thus it can be deduced that $S^{k}\le S^*$ for all iterations based on Lemma~\ref{le:stable}. By \eqref{barKl}, the SARE \eqref{S*} is reformulated as 
	\begin{align}\label{KlS*}
		&\big(A-B(K^*+\bar{K}^k)\big)^{\top}S^{*}+S^{*}\big(A-B(K^*+\bar{K}^k)\big)\nonumber\\
		=&\; Q_{\Gamma}^{*}-(\bar{K}^{k})^{\top}\!\Lambda^*\bar{K}^{k}\!+\!(\bar{K}^{*}\!-\bar{K}^{k})^{\top}\!\Lambda^*(\bar{K}^{*}\!-\bar{K}^{k}).
	\end{align}
	By Lemma~\ref{le:stable} and result (2), subtracting \eqref{KlPil} from \eqref{KlS*} yields 
	\begin{align*}
		&\big(A-B(K^*+\bar{K}^k)\big)^{\top}(S^{*}-S^k)\nonumber\\
		&+(S^{*}-S^k)\big(A-B(K^*+\bar{K}^k)\big)\nonumber\\
		=&\ Q_{\Gamma}^{*}-Q_{\Gamma}^{k+1}+(\bar{K}^{*}-\bar{K}^{k})^{\top}\Lambda^*(\bar{K}^{*}-\bar{K}^{k})\le 0
	\end{align*}
	for $k\ge 0$. Hence, $Q_{\Gamma}^{k+1}\ge Q_{\Gamma}^{*}$ holds, and result (1) is established. 
	
	Finally, we show that result (3) holds. From \eqref{Sl}, \eqref{barKl} and \eqref{Gammal}, the final values $S^*$, $Q_{\Gamma}^*$ and $\bar{K}^*$ satisfy \eqref{S*}, \eqref{S*k} and \eqref{barK*}, such that $\lim_{k\to \infty} Q_{\Gamma}^k=Q_{\Gamma}^*$, $\lim_{k\to \infty} S^k=S^*$, and $\lim_{k\to \infty} \bar{K}^k=\bar{K}^*$ hold. At the $(k+1)$th iteration, substituting \eqref{Gammal} into \eqref{Sl} yields
	\begin{align*}
		&(A-BK^*)^{\top}S^{k+1}+S^{k+1}(A-BK^*)\nonumber\\
		&-\bar{K}_e^{\top}B^{\top}S^{k+1}-S^{k+1}B\bar{K}_e\nonumber\\
		=&(A-BK^*)^{\top}S^{k}+S^{k}(A-BK^*)\nonumber\\
		&-(\bar{K}^k)^{\top}\Lambda^*\bar{K}^k-(\bar{K}_e)^{\top}\Lambda^*\bar{K}_e.
	\end{align*}
	As $k\to \infty$, $S^k=S^{k+1}\to S^*$, then the above equation becomes
	\begin{align*}
		(\bar{K}_e-\bar{K}^*)^{\top}\Lambda^*(\bar{K}_e-\bar{K}^*)=0.
	\end{align*}
	Since $\Lambda^*>0$, we derive that $\bar{K}^*=\bar{K}_e$. This completes the proof. \hfill $\square$
\end{pf}

The following theorem will discuss the connection between $R_e$, $Q_e$, $Q_{\Gamma_e}$, $P_e$, $S_e$ and $R$, $Q^*$, $Q_{\Gamma}^*$, $P^*$, $S^*$, and illustrate the non-uniqueness of the solution set.

\begin{thm}[Non-Uniqueness of Solutions]\label{nonunique}
	Given the matrix $R$ adopted in Algorithm~\ref{algorithm1}, let the matrices $R_e$, $Q_e$, $Q_{\Gamma_e}$, $P_e$, and $S_e$ used for the expert agents satisfy \eqref{P_e}, \eqref{u_e}, and \eqref{S_e}. If there exist symmetric matrices $Q_o$, $Q_{\Gamma_o}$, $P_o$, and $S_o$ satisfying
	\begin{align}
		&B^{\top}P_{o}+D^{\top}P_{o}C=(R_o+D^{\top}P_oD)\Lambda_e^{-1}(B^{\top}P_{e}\nonumber\\
		&+D^{\top}P_{e}C),\label{BPo}\\
		&B^{\top}S_{o}=(R_o+D^{\top}P_oD)\Lambda_e^{-1}B^{\top}S_{e}, \label{BSo}\\
		&A^{\top}P_{o}+P_{o}A+C^{\top}P_{o}C+Q_o\nonumber\\
		&-K_e^{\top}(R_o+D^{\top}P_oD)K_e=0\label{Po},\\
		&(A-BK_e)^{\top}S_{o}+S_{o}(A-BK_e)-Q_{\Gamma_o}\nonumber\\
		&-\bar{K}_e^{\top}(R_o+D^{\top}P_oD)\bar{K}_e=0\label{So}.
	\end{align}
	where $R_o=R-R_e$, and $K_e$, $\bar{K}_e$ are given from \eqref{u_e}. Then, any converged solutions $Q^*=Q_e+Q_o$, $Q_{\Gamma}^*=Q_{\Gamma_e}+Q_{\Gamma_o}$, $P^*=P_e+P_o$, and $S^*=S_e+S_o$ from Algorithm~\ref{algorithm1} satisfy \eqref{P*} and \eqref{S*}, and give $K_e$ and $\bar{K}_e$ by \eqref{K*} and \eqref{barK*}.
\end{thm}

\begin{pf}
	First, recall $K_e$ and $\bar{K}_e$ in \eqref{u_e}. Using $P^*=P_e+P_o$, $S^*=S_e+S_o$, and $R_o=R-R_e$, together with conditions \eqref{BPo} and \eqref{BSo}, the converged control gains are derived as
	\begin{align*}
		K^*=&(R+D^{\top}P^*D)^{-1}(B^{\top}P^*+D^{\top}P^*C)\cr
		=&(R+D^{\top}P^*D)^{-1}(B^{\top}P_{e}+D^{\top}P_{e}C)\cr
		&+(R+D^{\top}P^*D)^{-1}(R+D^{\top}P^*D\cr
		&-R_e-D^{\top}P_eD)\Lambda_e^{-1}(B^{\top}P_{e}+D^{\top}P_{e}C)\cr
		=&\Lambda_e^{-1}(B^{\top}P_{e}+D^{\top}P_{e}C)=K_e
	\end{align*}
	and 
	\begin{align*}
		\bar{K}^*=&(R+D^{\top}P^*D)^{-1}B^{\top}S^*\cr
		=&(R+D^{\top}P^*D)^{-1}B^{\top}S_{e}+(R+D^{\top}P^*D)^{-1}\cr
		&\times (R+D^{\top}P^*D
		-R_e-D^{\top}P_eD)\Lambda_e^{-1}B^{\top}S_{e}\cr
		=&\Lambda_e^{-1}B^{\top}S_{e}=\bar{K}_e.
	\end{align*}
	
	Combining \eqref{Po} with \eqref{P_e}, and \eqref{So} with \eqref{S_e} leads to
	\begin{align*}
		&A^{\top}(P_e+P_{o})+(P_e+P_{o})A+C^{\top}(P_e+P_{o})C\nonumber\\
		&-K_e^{\top}(R_e+R_o+D^{\top}(P_e+P_{o})D)K_e+Q_e+Q_o=0,\\
		&(A-BK_e)^{\top}(S_{e}+S_{o})+(S_{e}+S_{o})(A-BK_e)\nonumber\\
		&-\bar{K}_e^{\top}(R_e+R_o+D^{\top}\!(P_e+P_{o})D)\bar{K}_e\!-Q_{\Gamma_e}\!-Q_{\Gamma_o}=0.
	\end{align*}
	Substituting $R=R_o+R_e$, $P^*=P_e+P_o$, $S^*=S_e+S_o$, $Q^*=Q_e+Q_o$, and $Q_{\Gamma}^*=Q_{\Gamma_e}+Q_{\Gamma_o}$ into the above equations yields \eqref{P*} and \eqref{S*}. This completes the proof.    \hfill $\square$
\end{pf}

\section{Model-Free Integral Inverse RL}\label{3}
To eliminate the reliance on system dynamics in the model-based inverse RL framework, this section develops a model-free inverse RL algorithm by leveraging integral RL. Following the iterative structure of Algorithm \ref{algorithm1}, the proposed approach solves Problem \ref{IRL} using only offline trajectory data collected from learner agents and a small amount of expert data, without requiring prior knowledge of system matrices.

\subsection{Data-Driven Solutions}\label{3.1}
First, we derive a data-driven iterative equation based on integral RL to calculate $P^l$ and the corresponding gain $K^l$. The system dynamics \eqref{x} of the learner agents can be rewritten as
\begin{align}\label{xKe}
	dx_{i} = &\big[Ax_{i}-BK_ex_{i}+B(u_{i}+K_ex_{i})\big]dt\nonumber\\
	& + \big[Cx_{i}-DK_ex_{i}+D(u_{i}+K_ex_{i})\big]dw_i.
\end{align}
By \eqref{xKe}, utilizing Itô’s formula to ${x}_i^{\top}P^l{x}_i$ yields
\begin{align}\label{Pito}
	d\big(x_i^{\top}P^lx_i\big)=&2x_i^{\top}P^l\big((A-BK_e)x_{i}+B(u_{i}+K_ex_{i})\big)dt\nonumber\\
	&+\big((C-DK_e)x_{i}+D(u_{i}+K_ex_{i})\big)^{\top}P^l\nonumber\\
	&\times \big((C-DK_e)x_{i}+D(u_{i}+K_ex_{i})\big)dt\nonumber\\
	&+2x_i^{\top}P^l\big(Cx_{i}+Du_{i}\big)dw_i.
\end{align}
Substituting \eqref{Pl} into \eqref{Pito}, one has
\begin{align}\label{Pdiff}
	d\big(x_i^{\top}P^lx_i\big)=&\big[-x_i^{\top}(Q^l+K_e^{\top}RK_e)x_i+u_i^{\top}\Theta^lu_i\nonumber\\
	&-x_i^{\top}\!K_e^{\top}\!\Theta^lK_ex_i+2(u_i+K_ex_i)^{\top}\!\mathcal{B}_1^lx_i\big]dt\nonumber\\
	&+2x_i^{\top}P^l\big(Cx_{i}+Du_{i}\big)dw_i, 
\end{align}
where $\mathcal{B}_1^l\triangleq B^{\top}P^l+D^{\top}P^lC$ and $\Theta^l\triangleq D^{\top}P^lD$. Note that $K^l=(R+\Theta^l)^{-1}\mathcal{B}_1^l$. Let $T>0$ be an integral time interval. We integrate both sides of \eqref{Pdiff} from $t$ to $t+T$, and take the expectation to obtain
\begin{align}\label{Pintegrate}
	&\mathbb{E}\big[x_i(t+T)^{\top}P^lx_i(t+T)\big]-\mathbb{E}\big[x_i(t)^{\top}P^lx_i(t)\big]\nonumber\\
	&-2\mathbb{E}\bigg[\int_{t}^{t+T}(u_i+K_ex_i)^{\top}\mathcal{B}_1^lx_id\tau\bigg]\cr
	&-\mathbb{E}\bigg[\int_{t}^{t+T}u_i^{\top}\Theta^lu_id\tau\bigg]+\mathbb{E}\bigg[\int_{t}^{t+T}x_i^{\top}K_e^{\top}\Theta^lK_ex_id\tau\bigg]\cr
	=&-\mathbb{E}\bigg[\int_{t}^{t+T}x_i^{\top}(Q^l+K_e^{\top}RK_e)x_id\tau\bigg].
\end{align}
Equation \eqref{Pintegrate} is the data-driven version of \eqref{Pl} and \eqref{Kl} in Algorithm \ref{algorithm1}, and it avoids explicit reliance on the system matrix $A$. To eliminate the dependence on all remaining system parameters, a completely model-free implementation is presented in the next subsection.

Next, we derive a data-driven counterpart of the weight update equation \eqref{Ql}. Premultiplying and postmultiplying both sides of \eqref{Ql} by $x_i^\top$ and $x_i$, respectively, and rearranging terms yield
\begin{align}\label{xQlx}
	x_i^{\top}Q^{l+1}x_i=&-(Ax_i+Bu_i)^{\top}P^lx_i-x_i^{\top}P^l(Ax_i+Bu_i)\nonumber\\
	&-(Cx_i+Du_i)^{\top}\!P^l(Cx_i+Du_i)+u_i^{\top}\Theta^lu_i\nonumber\\
	&+x_i^{\top}\!(K^l)^{\top}\!(R+\Theta^l)K^lx_i+2u_i^{\top}\!\mathcal{B}_1^lx_i.
\end{align}
Based on \eqref{x}, taking the expectation and integrating both sides of \eqref{xQlx} over the time interval $[t,t+T]$ yields
\begin{align}\label{ExQlx}
	&\mathbb{E}\int_{t}^{t+T}x_i^{\top}Q^{l+1}x_id\tau\nonumber\\
	=&\ \mathbb{E}x_i(t)^{\top}P^lx_i(t)-\mathbb{E}x_i(t+T)^{\top}P^lx_i(t+T)\nonumber\\
	&+\mathbb{E}\int_{t}^{t+T}x_i^{\top}(K^l)^{\top}(R+\Theta^l)K^lx_id\tau\nonumber\\
	&+2\mathbb{E}\int_{t}^{t+T}u_i^{\top}\mathcal{B}_1^lx_id\tau+\mathbb{E}\int_{t}^{t+T}u_i^{\top}\Theta^lu_id\tau.
\end{align}
Now, the iteration procedure of the first loop in Algorithm~\ref{algorithm1} can be implemented solely using the measured trajectory data. Moreover, the final values used for the second loop still follow the definitions given in Algorithm~\ref{algorithm1}.

Second, we continue to develop the data-driven iterative forms for \eqref{Sl} and \eqref{barKl}. Rewrite the dynamics \eqref{x} of the learner agents and take the expectation on both sides to obtain
\begin{align}\label{xbarKe}
	d\bar{x}_{i} = &\big[\big(A-B(K^*+\bar{K}_e)\big)\bar{x}_{i}\nonumber\\
	&+B(K^*+\bar{K}_e)\bar{x}_{i}+B\bar{u}_{i}\big]dt,
\end{align}
where $\bar{x}_{i}=\mathbb{E}[x_i]$ and $\bar{u}_i=\mathbb{E}[u_i]$. Then, differentiating $\bar{x}_i^{\top}S^k\bar{x}_i$ with respect to $\bar{x}_i$ gives
\begin{align}\label{dS}
	d\big(\bar{x}_i^{\top}S^k\bar{x}_i\big)=&2\bar{x}_i^{\top}S^k\big[\big(A-B(K^*+\bar{K}_e)\big)\bar{x}_{i}\nonumber\\
	&+B(K^*+\bar{K}_e)\bar{x}_{i}+B\bar{u}_{i}\big]dt.
\end{align}
Integrating the above equation over $[t,t+T]$ and substituting \eqref{Sl}, we obtain
\begin{align}\label{Sintegrate}
	&\bar{x}_i(t+T)^{\top}S^k\bar{x}_i(t+T)-\bar{x}_i(t)^{\top}S^k\bar{x}_i(t)\nonumber\\
	&-2\int_{t}^{t+T}\big(\bar{u}_i+(K^*+\bar{K}_e)\bar{x}_i\big)^{\top}\mathcal{B}_2^k\bar{x}_id\tau\cr
	=&\int_{t}^{t+T}\bar{x}_i^{\top}\big(Q_{\Gamma}^k-\bar{K}_e^{\top}(R+\Theta^*)\bar{K}_e\big)\bar{x}_id\tau,
\end{align}
where $\mathcal{B}_2^k\triangleq B^{\top}S^k$, and $\Theta^*=\lim_{l\to \infty}\Theta^l$ has been obtained by the first loop in Algorithm~\ref{algorithm1}. Then \eqref{barKl}
can be calculated by $\bar{K}^k=(R+\Theta^*)^{-1}\mathcal{B}_2^k$. Similarly, the weight update equation \eqref{Gammal} can also be solved through measured trajectory data.
Premultiplying and postmultiplying both sides of \eqref{Gammal} by $\bar{x}_i^\top$ and $\bar{x}_i$, respectively, we have
\begin{align}\label{xGammalx}
	\bar{x}_i^{\top}Q_{\Gamma}^{k+1}\bar{x}_i=&(A\bar{x}_i+B\bar{u}_i)^{\top}S^k\bar{x}_i+\bar{x}_i^{\top}S^k (A\bar{x}_i+B\bar{u}_i)\nonumber\\
	&-2(\bar{u}_i+K^*\bar{x}_i)^{\top}\mathcal{B}_2^k\bar{x}_i\nonumber\\
	&-\bar{x}_i^{\top}(\bar{K}^k)^{\top}(R+\Theta^*)\bar{K}^k\bar{x}_i.
\end{align}
From \eqref{xbarKe}, we deduce that $A\bar{x}_i+B\bar{u}_i=\frac{d\bar{x}_i}{dt}$. Substituting this relation into \eqref{xGammalx} and integrating both sides over $[t,t+T]$ yields
\begin{align}\label{ExGammalx}
	&\int_{t}^{t+T}\bar{x}_i^{\top}Q_{\Gamma}^{k+1}\bar{x}_id\tau\nonumber\\
	=&\bar{x}_i(t+T)^{\top}S^k\bar{x}_i(t+T)-\bar{x}_i(t)^{\top}S^k \bar{x}_i(t)\nonumber\\
	&-2\int_{t}^{t+T}(\bar{u}_i+K^*\bar{x}_i)^{\top}\mathcal{B}_2^k\bar{x}_id\tau\nonumber\\
	&-\int_{t}^{t+T} \bar{x}_i^{\top}(\bar{K}^k)^{\top}(R +\Theta^*) \bar{K}^k\bar{x}_i d\tau.
\end{align}
To summarize, we have derived the data-driven iterative equations \eqref{Pintegrate}, \eqref{ExQlx}, \eqref{Sintegrate}, and \eqref{ExGammalx}, which can replace the model-based counterparts in Algorithm~\ref{algorithm1}.

\subsection{Model-Free Implementation}\label{3.2}
A model-free implementation method for obtaining the data-driven solutions using only measured trajectory data is shown below. For a matrix $G\in \mathbb{R}^{m\times n}$, a symmetric matrix $M\in \mathbb{S}^{n}$ and a vector $x\in \mathbb{R}^n$, we first define the following notation operators:
$\mathrm{vec}(G)\triangleq [g_{11}, g_{21},\dots, g_{m1}, g_{12}, g_{22},\dots, g_{mn}]^{\top} \in \mathbb{R}^{mn}$,
$\mathrm{vecs}(M)\!\triangleq\![m_{11},2m_{12},\! \dots, 2m_{1n}, m_{22}, 2m_{23},\! \dots, 2m_{n-1,n}, m_{nn}]^{\top} \!\in \mathbb{R}^{\frac{n(n+1)}{2}}$,
$\tilde{x}\triangleq [x_{1}^2, x_1x_2,\!\dots, x_1x_n, x_2^2, \!\dots, x_{n-1}x_n, x_n^2]^{\top} \!\in  \mathbb{R}^{\frac{n(n+1)}{2}}$.
Let $G=[g_1^{\top},g_2^{\top},\dots,g_m^{\top}]^{\top} \in \mathbb{R}^{m\times n}$. For $i,j=1,2,\dots,m$, we further define 
\begin{align*}
	&\mu_{ij} \triangleq g_i^{\top} \otimes g_j^{\top} \in \mathbb{R}^{n^2},\\
	&\check{G}\!\triangleq \![\mu_{11}, \!\dots, \mu_{1m}, \mu_{22}, \!\dots, \mu_{(m-1)m}, \mu_{mm}]\!\in \!\mathbb{R}^{n^2\times \!\frac{m(m+1)}{2}}.
\end{align*}
To further solve for $P^l$ and $K^l$ from the data-driven equation \eqref{Pintegrate}, we utilize the Kronecker product and vectorization theory, and define the following matrices as
\begin{equation*}
	\left\{
	\begin{aligned}
		\Delta_{\tilde{x}_i\tilde{x}_i}&\triangleq \mathbb{E}[\tilde{x}_i(t+T)-\tilde{x}_i(t), \dots,\\ 
		&\qquad\ \tilde{x}_i(t+sT)-\tilde{x}_i(t+(s-1)T)]^{\top},\\
		\mathcal{I}_{x_ix_i}&\triangleq \mathbb{E}\bigg[\int_{t}^{t+T}x_i(\tau)\otimes x_i(\tau)d\tau, \dots,\\ &\qquad\ \int_{t+(s-1)T}^{t+sT}x_i(\tau)\otimes x_i(\tau)d\tau\bigg]^{\top},\\
		\mathcal{I}_{x_iu_i}&\triangleq \mathbb{E}\bigg[\int_{t}^{t+T}x_i(\tau)\otimes u_i(\tau)d\tau, \dots,\\&\qquad\ \int_{t+(s-1)T}^{t+sT}x_i(\tau)\otimes u_i(\tau)d\tau\bigg]^{\top},\\
		\mathcal{I}_{\tilde{u}_i}&\triangleq \mathbb{E}\bigg[\int_{t}^{t+T}\tilde{u}_i(\tau)d\tau, \dots, \int_{t+(s-1)T}^{t+sT}\tilde{u}_i(\tau)d\tau\bigg]^{\top},
	\end{aligned}
	\right.
\end{equation*}
where $s$ denotes the number of time intervals from $t$ to $t+sT$. Then based on the above definitions, equation \eqref{Pintegrate} can be derived as
\begin{align*}
	\Phi_p\begin{bmatrix}
		\mathrm{vecs}(P^l)\\
		\mathrm{vec}(\mathcal{B}_1^l)\\
		\mathrm{vecs}(\Theta^l)
	\end{bmatrix}=\Xi_p^l,
\end{align*}
where $\Phi_p=[\Delta_{\tilde{x}_i\tilde{x}_i},-2\mathcal{I}_{x_ix_i}(I_n\otimes K_e^{\top})-2\mathcal{I}_{x_iu_i},\mathcal{I}_{x_ix_i}\check{K}_e-\mathcal{I}_{\tilde{u}_i}]$, and $\Xi_p^l=-\mathcal{I}_{x_ix_i}\mathrm{vec}(Q^l+K_e^{\top}RK_e)$. 
Define \[\mathcal{I}_{\tilde{x}_i}\triangleq \mathbb{E}\bigg[\int_{t}^{t+T}\tilde{x}_i(\tau)d\tau, \dots, \int_{t+(s-1)T}^{t+sT}\tilde{x}_i(\tau)d\tau\bigg]^{\top},\]
then based on the above definitions and obtained parameters, \eqref{ExQlx} has a linear matrix form as
\begin{align*}
	\Phi_q \mathrm{vecs}(Q^{l+1})=\Xi_q^l,
\end{align*}
where $\Phi_q=\mathcal{I}_{\tilde{x}_i}$, and $\Xi_q^l=-\Delta_{\tilde{x}_i\tilde{x}_i}\mathrm{vecs}(P^l)+\mathcal{I}_{x_ix_i}\check{K}^l \mathrm{vecs}(R+\Theta^l)+2\mathcal{I}_{x_iu_i}\mathrm{vec}(\mathcal{B}_1^l)+\mathcal{I}_{\tilde{u}_i}\mathrm{vecs}(\Theta^l)$. 
\begin{lem}\label{L3.1}
	If there exists $s_1>0$, such that for all $s\ge s_1$,
	\begin{align}
		&\mathrm{rank}([\mathcal{I}_{{x}_i{x}_i},\mathcal{I}_{{x}_i{u}_i},\mathcal{I}_{\tilde{u}_i}])=\frac{n(n+1)}{2}+mn+\frac{m(m+1)}{2},\label{rank1.1}\\
		&\mathrm{rank}(\mathcal{I}_{\tilde{x}_i})=\frac{n(n+1)}{2}.\label{rank1.2}
	\end{align}
	Then using the least squares method, $P^l$, $\mathcal{B}_1^l$, $\Theta^l$ and ${Q}^{l+1}$ can be calculated by
	\begin{align}
		\begin{bmatrix}
			\mathrm{vecs}(P^l)\\
			\mathrm{vec}(\mathcal{B}_1^l)\\
			\mathrm{vecs}(\Theta^l)
		\end{bmatrix}&=(\Phi_p^{\top}\Phi_p)^{-1}\Phi_p^{\top}\Xi_p^l,\label{data1.1}\\
		\mathrm{vecs}(Q^{l+1})&=(\Phi_q^{\top}\Phi_q )^{-1}\Phi_q^{\top}\Xi_q^l.\label{data1.2}
	\end{align}
	Moreover, the gain $K^l$ is given by $K^l=(R+\Theta^l)^{-1}\mathcal{B}_1^l$. 
\end{lem}

\begin{pf}
	First, under the rank condition \eqref{rank1.1}, we show by contradiction that $\Phi_p$ has full column rank, so that the solution to \eqref{data1.1} is unique.
	
	Suppose that $\Phi_pH_p=0$, where $H_p\!=\![\mathrm{vecs}(X)^{\top}\!,\!\mathrm{vec}(Y)^{\top}\!,\\ \mathrm{vecs}(Z)^{\top}]^{\top}$ is a nonzero vector, the symmetric matrices $X, Z\in \mathbb{S}^{n}$, and the matrix $Y \in \mathbb{R}^{m\times n}$. Then one has
	\begin{align}\label{Hp1}
		&\Delta_{\tilde{{x}}_i\tilde{{x}}_i}\mathrm{vecs}(X)-\mathcal{I}_{{x}_i{x}_i}\mathrm{vec}(K_e^{\top}Y+
		Y^{\top}K_e-K_e^{\top}ZK_e)\nonumber\\
		&-2\mathcal{I}_{{x}_i{u}_i}\mathrm{vec}(Y)-\mathcal{I}_{\tilde{u}_i}\mathrm{vecs}(Z)=0.
	\end{align}
	By \eqref{Pito} and \eqref{Pintegrate}, we have
	\begin{align*}
		&\Delta_{\tilde{{x}}_i\tilde{{x}}_i}\mathrm{vecs}(X)\nonumber\\
		=&\ \mathcal{I}_{{x}_i{x}_i}\mathrm{vec}\big((A-BK_e)^{\top}X+X(A-BK_e)\nonumber\\
		&+(C-DK_e)^{\top}X(C-DK_e)+K_e^{\top}B^{\top}X+XBK_e\nonumber\\
		&+K_e^{\top}D^{\top}XC+C^{\top}XDK_e-K_e^{\top}D^{\top}XDK_e\big)\nonumber\\
		&+2\mathcal{I}_{{x}_i{u}_i}\mathrm{vec}(B^{\top}X+D^{\top}XC)+\mathcal{I}_{\tilde{u}_i}\mathrm{vecs}(D^{\top}XD).
	\end{align*}
	Then substituting the above equation into \eqref{Hp1} yields
	\begin{align}\label{Hp3}
		\mathcal{I}_{{x}_i{x}_i}\mathrm{vec}(M_1)+2\mathcal{I}_{{x}_i{u}_i}\mathrm{vec}(M_{2})+\mathcal{I}_{\tilde{u}_i}\mathrm{vecs}(M_3)=0,
	\end{align}
	where $M_1=(A-BK_e)^{\top}X+X(A-BK_e)+(C-DK_e)^{\top}X(C-DK_e)+K_e^{\top}M_2+M_2^{\top}K_e-K_e^{\top}M_3K_e$, $M_2=B^{\top}X+D^{\top}XC-Y$, $M_3=D^{\top}XD-Z$.
	Under the rank condition \eqref{rank1.1}, equation \eqref{Hp3} implies that 
	\begin{equation}\label{M=0}
		M_1=M_2=M_3=0.
	\end{equation}
	Since $K_e$ is a stabilizer of the expert agent system, we can get
	\begin{align*}
	    &(A-BK_e)^{\top}X+X(A-BK_e)\nonumber\\
	    &+(C-DK_e)^{\top}X(C-DK_e)=0,
	\end{align*}
	such that $X=0$. Then based on \eqref{M=0}, it is further deduced that $Y=Z=0$.
	Thus, the result $H_p = 0$ leads to a contradiction with the assumption that $H_p\ne 0$, which indicates that $\Phi_p$ must have full column rank. Since the rank condition \eqref{rank1.2} holds, $\Phi_q $ has full column rank, and equation \eqref{data1.2} has a unique solution.
	This completes the proof.   \hfill $\square$
\end{pf}
To implement the data-driven iterative equation \eqref{Sintegrate}, we define
\begin{equation*}
	\left\{
	\begin{aligned}
		\Delta_{\tilde{\bar{x}}_i\tilde{\bar{x}}_i}\triangleq& [\tilde{\bar{x}}_i(t+T)-\tilde{\bar{x}}_i(t), \dots,\\
		&\ \tilde{\bar{x}}_i(t+sT)-\tilde{\bar{x}}_i(t+(s-1)T)]^{\top},\\
		\mathcal{I}_{\bar{x}_i\bar{x}_i}\triangleq& \bigg[\int_{t}^{t+T}\bar{x}_i(\tau)\otimes \bar{x}_i(\tau)d\tau, \dots,\\&\ \int_{t+(s-1)T}^{t+sT}\bar{x}_i(\tau)\otimes \bar{x}_i(\tau)d\tau\bigg]^{\top},\\
		\mathcal{I}_{\bar{x}_i\bar{u}_i}\triangleq& \bigg[\int_{t}^{t+T}\bar{x}_i(\tau)\otimes \bar{u}_i(\tau)d\tau, \dots,\\&\ \int_{t+(s-1)T}^{t+sT}\bar{x}_i(\tau)\otimes \bar{u}_i(\tau)d\tau\bigg]^{\top},
	\end{aligned}
	\right.
\end{equation*}
and equation \eqref{Sintegrate} can be transformed into a linear matrix form as
\begin{align*}
	\Psi_{s}\begin{bmatrix}
		\mathrm{vecs}(S^k)\\
		\mathrm{vec}(\mathcal{B}_2^k)
	\end{bmatrix}=\Upsilon_{s}^k,
\end{align*}
where $\Psi_{s}=\big[\Delta_{\tilde{\bar{x}}_i\tilde{\bar{x}}_i},-2\mathcal{I}_{\bar{x}_i\bar{x}_i}\big(I_n\otimes (K^*+\bar{K}_e)^{\top}\big)-2\mathcal{I}_{\bar{x}_i\bar{u}_i}\big]$, and $\Upsilon_{s}^k=\mathcal{I}_{\bar{x}_i\bar{x}_i}\mathrm{vec}\big(Q_{\Gamma}^k-\bar{K}_e^{\top}(R+\Theta^*)\bar{K}_e\big)$.
For \eqref{ExGammalx}, we define 
\[
\mathcal{I}_{\tilde{\bar{x}}_i}\triangleq \bigg[\int_{t}^{t+T}\tilde{\bar{x}}_i(\tau)d\tau, \dots, \int_{t+(s-1)T}^{t+sT}\tilde{\bar{x}}_i(\tau)d\tau\bigg]^{\top}.
\]
Then equation \eqref{ExGammalx} has a linear matrix form as 
\begin{align*}
	\Psi_{\gamma}\mathrm{vecs}(Q_{\Gamma}^{k+1})=\Upsilon_{\gamma}^k,
\end{align*}
where $\Psi_{\gamma}=\mathcal{I}_{\tilde{\bar{x}}_i}$, and $\Upsilon_{\gamma}^k=\Delta_{\tilde{\bar{x}}_i\tilde{\bar{x}}_i}\mathrm{vecs}(S^k)-2\big[\mathcal{I}_{\bar{x}_i\bar{u}_i}+\mathcal{I}_{\bar{x}_i\bar{x}_i}\big(I_n\otimes (K^*)^{\top}\big)\big]\mathrm{vec}(\mathcal{B}_2^k)-\mathcal{I}_{\bar{x}_i\bar{x}_i}\check{\bar{K}}^k\mathrm{vecs}(R+\Theta^*)$.
\begin{lem}\label{L3.2}
	If there exists $s_2>0$, such that for all $s\ge s_2$,
	\begin{align}
		&\mathrm{rank}([\mathcal{I}_{\bar{x}_i\bar{x}_i},\mathcal{I}_{\bar{x}_i\bar{u}_i}])=\frac{n(n+1)}{2}+mn,\label{rank2.1}\\
		&\mathrm{rank}(\mathcal{I}_{\tilde{\bar{x}}_i})=\frac{n}{2}(n+1).\label{rank2.2}
	\end{align}
	Then using the least squares method, $S^k$, $\mathcal{B}_2^k$ and $Q_{\Gamma}^{k+1}$ can be computed by
	\begin{align}
		\begin{bmatrix}
			\mathrm{vecs}(S^k)\\
			\mathrm{vec}(\mathcal{B}_2^k)
		\end{bmatrix}&=(\Psi_{s}^{\top}\Psi_{s})^{-1}\Psi_{s}^{\top}\Upsilon_{s}^k,\label{data2.1}\\
		\mathrm{vecs}(Q_{\Gamma}^{k+1})&=(\Psi_{\gamma}^{\top}\Psi_{\gamma} )^{-1}\Psi_{\gamma}^{\top}\Upsilon_{\gamma}^k,\label{data2.2}
	\end{align}
	and $\bar{K}^k$ can be obtained by $\bar{K}^k=(R+\Theta^*)^{-1}\mathcal{B}_2^k$.
\end{lem}
\begin{pf}
	Similar to the proof of Lemma \ref{L3.1}, we adopt a contradictory approach to show that $\Psi_{s}$ has full column rank. Let $H_{s}=[\mathrm{vecs}(X)^{\top},\mathrm{vec}(Y)^{\top}]^{\top}$ be a nonzero vector, and
	suppose that $\Psi_{s}H_{s}=0$. Then we have
	\begin{align}\label{H1}
		&\Delta_{\tilde{\bar{x}}_i\tilde{\bar{x}}_i}\mathrm{vecs}(X)-\mathcal{I}_{\bar{x}_i\bar{x}_i}\mathrm{vec}[(K^*+\bar{K}_e)^{\top}Y+\nonumber\\
		&Y^{\top}(K^*+\bar{K}_e)]-2\mathcal{I}_{\bar{x}_i\bar{u}_i}\mathrm{vec}(Y)=0.
	\end{align}
	By \eqref{dS} and \eqref{Sintegrate}, we have
	\begin{align}\label{H2}
		&\Delta_{\tilde{\bar{x}}_i\tilde{\bar{x}}_i}\mathrm{vecs}(X)\nonumber\\
		=&\ \mathcal{I}_{\bar{x}_i\bar{x}_i}\mathrm{vec}\big[\big(A-B(K^*+\bar{K}_e)\big)^{\top}X\nonumber\\
		&+X\big(A-B(K^*+\bar{K}_e)\big)+(K^*+\bar{K}_e)^{\top}B^{\top}X\nonumber\\
		&+X^{\top}B(K^*+\bar{K}_e)\big]+2\mathcal{I}_{\bar{x}_i\bar{u}_i}\mathrm{vec}(B^{\top}X).
	\end{align}
	Then substituting \eqref{H2} into \eqref{H1} yields
	\begin{align}\label{H3}
		\mathcal{I}_{\bar{x}_i\bar{x}_i}\mathrm{vec}(N_1)+2\mathcal{I}_{\bar{x}_i\bar{u}_i}\mathrm{vec}(N_2)=0,
	\end{align}
	where $N_1=\big(A-B(K^*+\bar{K}_e)\big)^{\top}X+X\big(A-B(K^*+\bar{K}_e)\big)+(K^*+\bar{K}_e)^{\top}N_2+N_2^{\top}(K^*+\bar{K}_e)$ and $N_2=B^{\top}X-Y$.
	Under the rank condition \eqref{rank2.1}, equation \eqref{H3} implies that $N_1=N_2=0$. Since $A-B(K^*+\bar{K}_e)$ is Hurwitz, then $N_1=\big(A-B(K^*+\bar{K}_e)\big)^{\top}X+X\big(A-B(K^*+\bar{K}_e)\big)=0$ has a unique solution $X=0$, such that $Y=0$. This contradicts the assumption with $H_{s}\ne 0$, so $\Psi_{s}$ has full column rank. Based on the rank condition \eqref{rank2.2}, $\Psi_{\gamma}$ also has full column rank, that is, equation \eqref{data2.2} has a unique solution. This completes the proof. \hfill $\square$
\end{pf}

\begin{rem}
	To satisfy the rank conditions \eqref{rank1.1}, \eqref{rank1.2}, \eqref{rank2.1}, and \eqref{rank2.2}, $s$ must be not less than $\frac{n}{2}(n+1)+mn+\frac{m}{2}(m+1)$. This can be guaranteed by adding exploration noise to the control input for enriching the collected system data. In implementation, the control input is constructed as $u_i=-Kx_i-\bar{K}\bar{x}+e_i$, where $e_i(t)$ is the exploration noise. Common noise types in the literature include bounded random noise \citep{Al-Tamimi2007}, exponentially decaying probing noise \citep{Vamvoudakis2011}, and multi-frequency sinusoidal excitation \citep{Jiang2012}. In this paper, the last one is adopted for the subsequent simulation validation.	
\end{rem}
Based on Lemmas~\ref{L3.1} and \ref{L3.2}, we develop a model-free inverse RL algorithm (See Algorithm~\ref{algorithm2}) for solving Problem~\ref{IRL}. The proposed algorithm requires no prior knowledge of system dynamics, relying solely on observed trajectory data. 
\begin{algorithm}[!t]
	\caption{A model-free inverse RL algorithm for indefinite MF social optimization.}\label{algorithm2}
	\begin{algorithmic}[1]
		\STATE \textbf{Initialization:} Select an arbitrary matrix $R$, suitable initial weights $Q^0$, $\Gamma^0$, and control gains $K'$, $\bar{K}'$. Let $Q_{\Gamma}^0=(\Gamma^0)^{\top}Q^0+Q^0\Gamma^0-(\Gamma^0)^{\top}Q^0\Gamma^0$, and calculate $\tilde{K}_e$ by \eqref{K_e}.    
		\STATE \textbf{Data collection:} Apply $u_i=-K'x_i-\bar{K}'\bar{x}+e_i$ to each learner agent \eqref{x}, and collect enough state and input data from learner trajectories to meet the rank conditions.
		\STATE \textbf{Loop 1:} Set $l=0$, and $\epsilon_q$ as a small threshold.
		\REPEAT
		\STATE \textbf{Policy evaluation and improvement:} Calculate $P^l$, $\mathcal{B}_1^l$ and $\Theta^l$ by equation \eqref{data1.1}, and update $K^l$ by $K^l=(R+\Theta^l)^{-1}\mathcal{B}_1^l$.
		\STATE \textbf{Weight improvement:} Update $Q^{l+1}$ by equation \eqref{data1.2}.
		\STATE Let $l=l+1$.
		\UNTIL{$\|Q^{l}-Q^{l-1}\|\le \epsilon_q$.}
		\STATE Let the final values $K^* = K^{l-1}$, $P^* = P^{l-1}$, $Q^* = Q^{l}$, and $\Theta^* = \Theta^{l-1}$.
		\STATE \textbf{Loop 2:} Set $k=0$, and $\epsilon_{\gamma}$ as a small threshold. 
		\REPEAT
		\STATE \textbf{MF policy evaluation and improvement:} Calculate $S^{k}$ and $\mathcal{B}_2^k$ using equation \eqref{data2.1}, and update $\bar{K}^k$ by $\bar{K}^{k}= (R+\Theta^*)^{-1}\mathcal{B}_2^k$.
		\STATE \textbf{MF weight improvement:} Update $Q_{\Gamma}^{k+1}$ by equation \eqref{data2.2}.
		\STATE Let $k=k+1$.
		\UNTIL{$\|Q_{\Gamma}^{k}-Q_{\Gamma}^{k-1}\|\le \epsilon_{\gamma}$.}
		\STATE Let final values $\bar{K}^* = \bar{K}^{k-1}$, $S^* = S^{k-1}$, and $Q_{\Gamma}^* = Q_{\Gamma}^{k}$.
		Moreover, $\Pi^*=P^*+S^*$. Applying $u_{i}= -K^*x_{i}-\bar{K}^*\bar{x}$ to learner agents enables them to imitate the trajectories of the expert population system.	
	\end{algorithmic}
\end{algorithm}

Finally, we conduct the convergence analysis of Algorithm \ref{algorithm2} to theoretically validate its effectiveness.
\begin{thm}\label{convergence2}
	Based on Lemmas~\ref{L3.1} and \ref{L3.2}, as $l,k \to \infty$, $P^l$, $K^l$, $Q^l$, $S^k$, $\bar{K}^k$, and $Q_{\Gamma}^k$ in Algorithm \ref{algorithm2} converge to $P^*$, $K^*$, $Q^*$, $S^*$, $\bar{K}^*$, and $Q_{\Gamma}^*$, respectively, as defined in equations \eqref{P*}--\eqref{barK*}.
\end{thm}
\begin{pf}
	Under the rank conditions \eqref{rank1.1}, \eqref{rank1.2}, \eqref{rank2.1}, and \eqref{rank2.2} in Lemmas~\ref{L3.1} and \ref{L3.2}, the iterative equations \eqref{data1.1}, \eqref{data1.2}, \eqref{data2.1}, and \eqref{data2.2} admit unique solutions. Thus, the data-driven equations in Algorithm \ref{algorithm2} are equivalent to the model-based equations in Algorithm \ref{algorithm1}. Therefore, the convergence result of Algorithm \ref{algorithm2} follows from Theorems \ref{convergence_P} and \ref{convergence_S}. This completes the proof. \hfill $\square$
\end{pf}

\section{Simulation}\label{4}
In this section, numerical simulations are provided to validate the effectiveness of the model-based and model-free inverse RL algorithms, namely Algorithms \ref{algorithm1} and \ref{algorithm2}. 

\textbf{Example 1.} The expert and learner populations each contain 40 homogeneous agents. The coefficients of the system \eqref{x_e} are
\vspace{-0.9em}
\[A=\begin{bmatrix}
	1 & 1\\
	3 & 2
\end{bmatrix},\
B=\begin{bmatrix}
	1\\
	2
\end{bmatrix},\ 
C=\begin{bmatrix}
	0.05& 0.03\\
	0.05& 0.02
\end{bmatrix},\
D=\begin{bmatrix}
	0.02\\
	0.01
\end{bmatrix},\]
and the noise term $w_i\sim \mathcal{N}(0,0.1)$.
The weight matrices for the expert social cost function \eqref{J_e}, and the solutions to SAREs \eqref{P_e} and \eqref{S_e} are 
\vspace{-1em}
\begin{align*}
	Q_e&=\begin{bmatrix}
		-1 & 0\\
		0 & 1
	\end{bmatrix},\ \Gamma_e=\begin{bmatrix}
		1 & 0\\
		0 & 1
	\end{bmatrix},\ R_e=1, \\
	P_e&=\begin{bmatrix}
		0.6524 & 0.8106\\
		0.8106 & 0.8019
	\end{bmatrix},\ S_e=\begin{bmatrix}
		0.3759 & -0.0220\\
		-0.0220 & -0.1962
	\end{bmatrix}.
\end{align*}
Besides, the optimal control gains are given by $K_e=[2.2744\ \: 2.4139]$ and 
$\bar{K}_e=[0.3317\ \: -0.4140]$. For generalization, we set $R=0.8$ for the learner, which differs from $R_e$, and select the initial weights as
\vspace{-1em}
\[Q^0=\begin{bmatrix}
	-4.0322 & 0\\
	0 & 0.6786
\end{bmatrix},\ \Gamma^0=\begin{bmatrix}
	-2.8836 & 0.4736\\
	-0.2079 & 0.2599
\end{bmatrix}.\] 
First, we solve the inverse RL problem using Algorithm~\ref{algorithm1} with stopping criteria $\epsilon_q=\epsilon_\gamma=10^{-6}$. After convergence, the learned gains are $K^*=[2.2738\ \: 2.4130]$ and $\bar{K}^*=[0.3308\ \: -0.4147]$. The Frobenius norm errors relative to the expert gains are $\|K^*-K_e\|_F=\|\bar{K}^*-\bar{K}_e\|_F=0.0011$. 
Fig.~\ref{fig1} shows the convergence of $Q^l$, $K^l$, $P^l$, $Q_{\Gamma}^k$, $\bar{K}^k$ and $S^k$ using the model-based inverse RL algorithm. Fig.~\ref{fig2} shows that the learner closely matches the expert trajectories under the gains learned from Algorithm~\ref{algorithm1}.
\begin{figure}[!t]
	\centering
	\begin{subfigure}{0.22\textwidth}
		\centering
		\includegraphics[height=2cm,width=4cm]{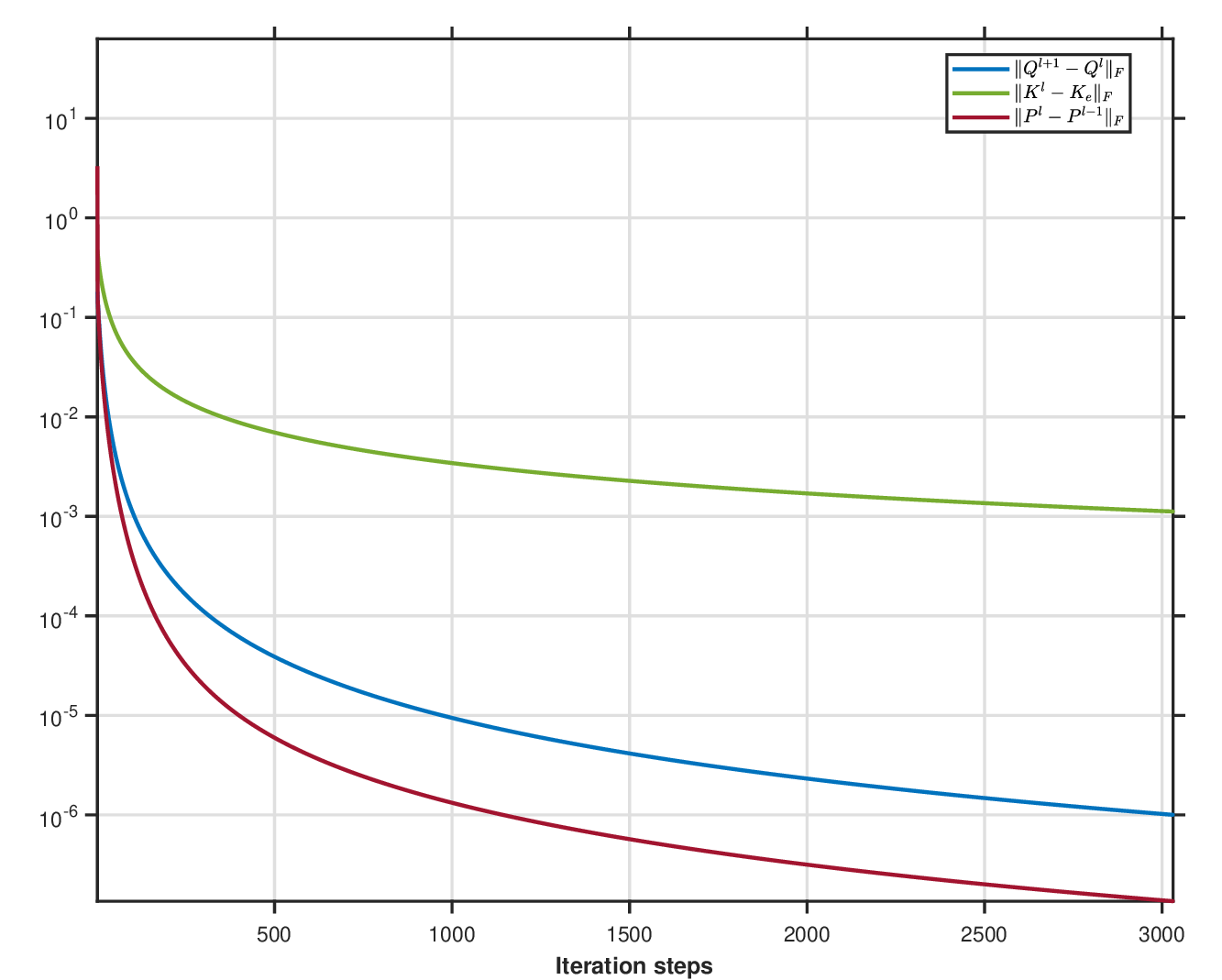} 
		\label{fig1:a}
	\end{subfigure}
	\hfill 
	\begin{subfigure}{0.22\textwidth}
		\centering
		\includegraphics[height=2cm,width=4cm]{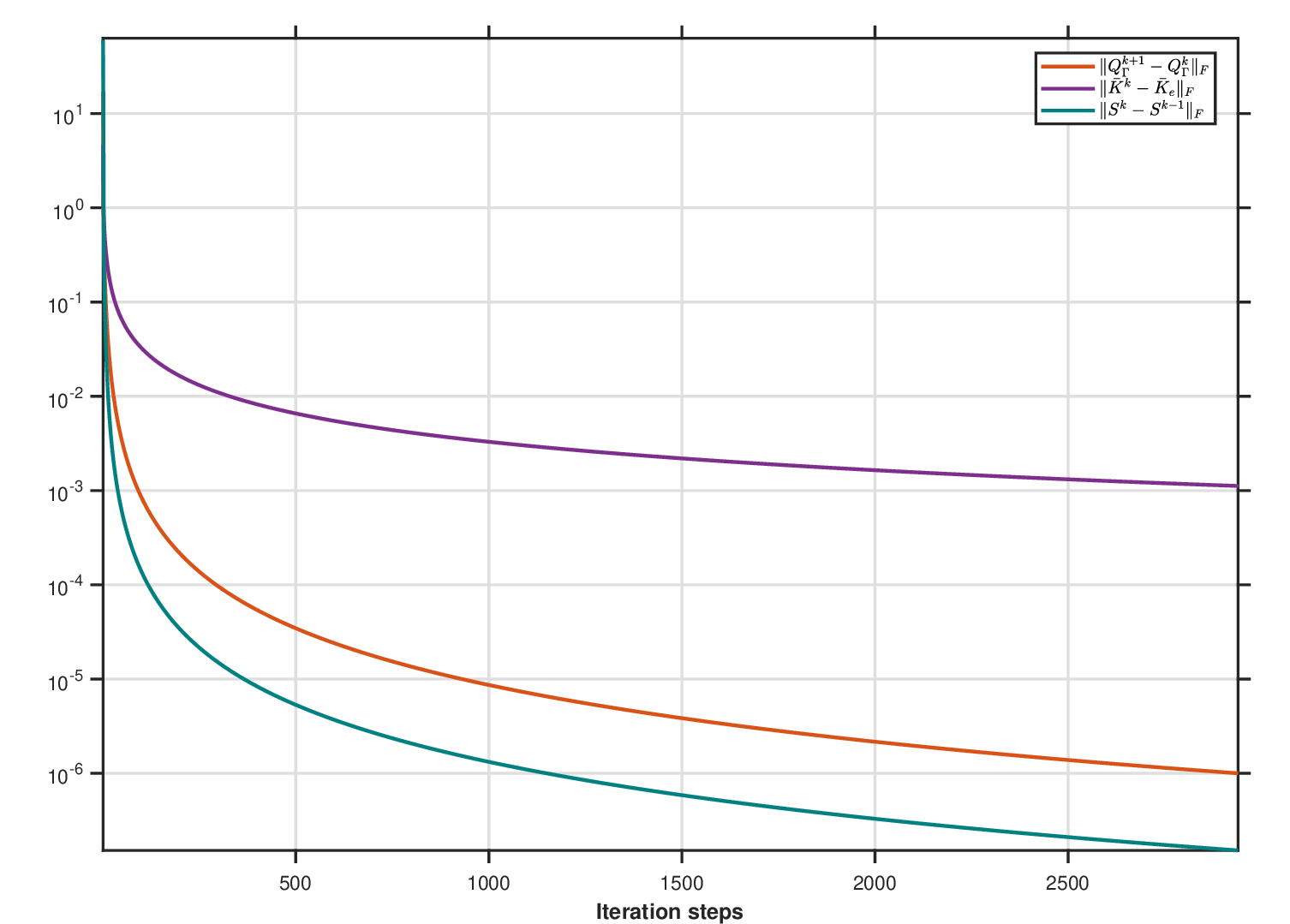} 
		\label{fig1:b}
	\end{subfigure}
    \vspace{-4pt}
	\caption{Convergence of $Q^l$, $K^l$, $P^l$, $Q_{\Gamma}^k$, $\bar{K}^k$ and $S^k$ using Algorithm~\ref{algorithm1}.}
	\label{fig1}
\end{figure}
\begin{figure}[!t]
	\centering
	\begin{subfigure}{0.22\textwidth}
		\centering
		\includegraphics[height=2cm,width=4cm]{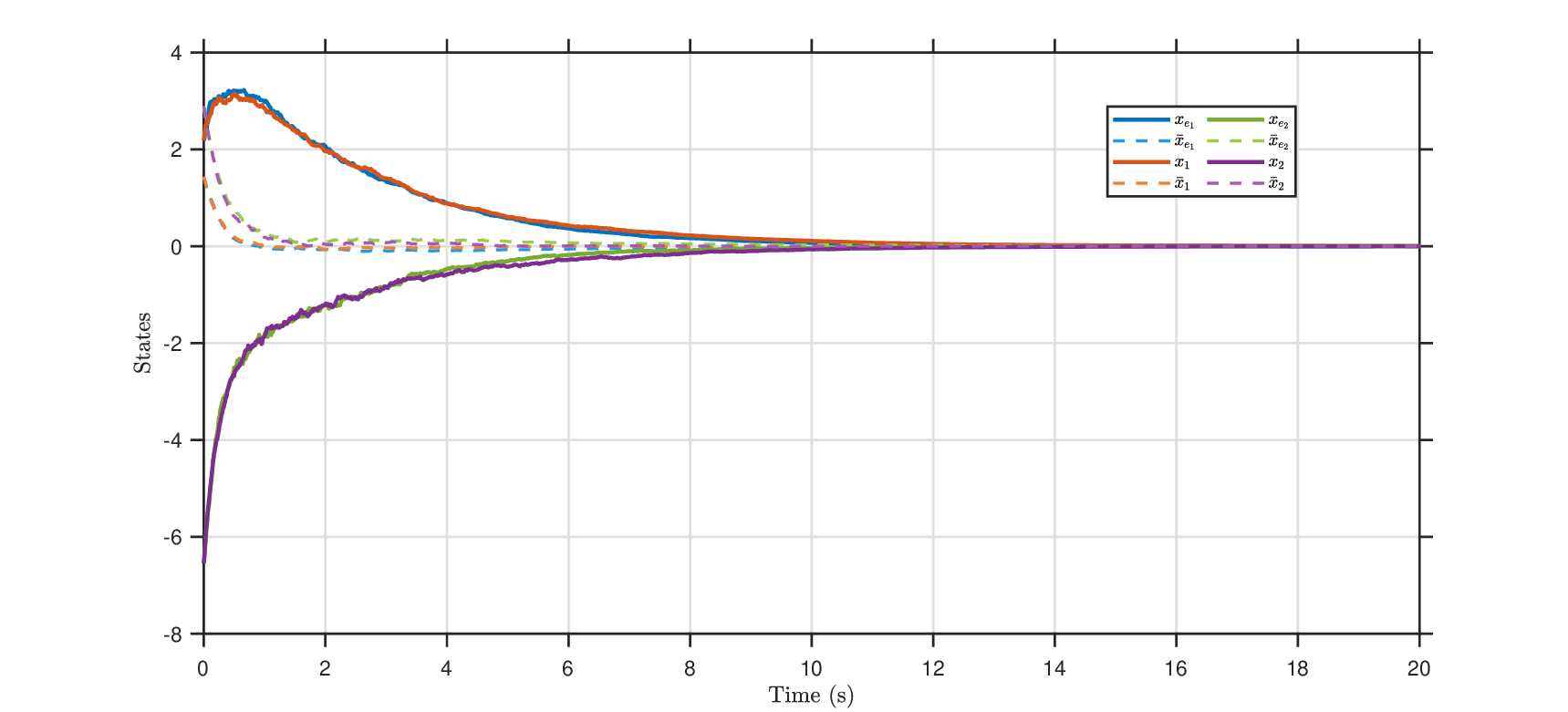} 
		\label{fig2:a}
	\end{subfigure}
	\hfill 
	\begin{subfigure}{0.22\textwidth}
		\centering
		\includegraphics[height=2cm,width=4cm]{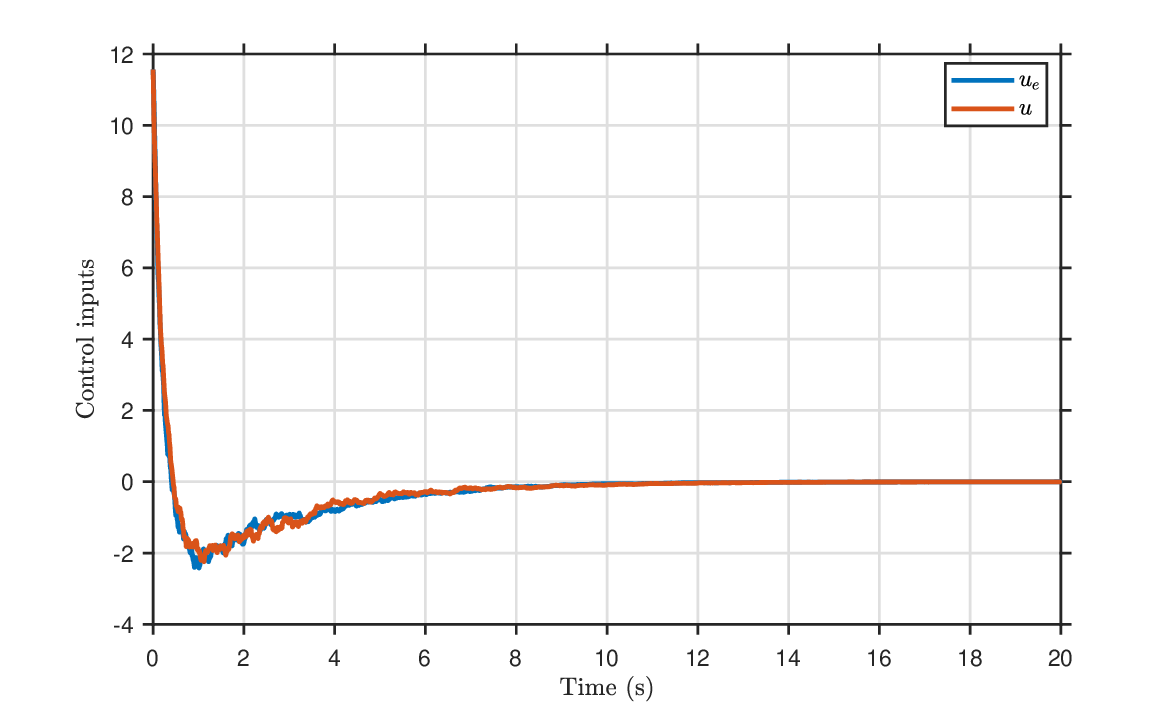} 
		\label{fig2:b}
	\end{subfigure}
    \vspace{-4pt}
	\caption{Trajectory imitation performance of the learner using Algorithm~\ref{algorithm1}.}
	\label{fig2}
\end{figure}

Next, we implement Algorithm~\ref{algorithm2} and present the simulation results. The initialization parameters are the same as those in Algorithm~\ref{algorithm1}. Then $u_i=[-0.9098\:\ -0.9656]x_i+[-0.0663\:\ 0.0828]\bar{x}+e$
are adopted to collect trajectory data within 1 second, where the exploration noise $e=2\sum_{j=1}^{100}sin(\omega_jt)$ with $\omega_j\in[-500,500]$ guarantees that the required rank conditions hold. When $\|Q^{l}-Q^{l-1}\|\le 10^{-5}$ and $\|Q_{\Gamma}^{k}-Q_{\Gamma}^{k-1}\|\le 10^{-5}$, the control gains converge to $K^*=[2.2679\:\ 2.4112]$ and $\bar{K}^*=[0.3302\:\ -0.4165]$ with $\|K^*-K_e\|_F=0.0070$ and $\|\bar{K}^*-\bar{K}_e\|_F=0.0029$. The norm errors rise slightly due to estimation errors, yet convergence remains favorable. Fig.~\ref{fig3} illustrates the convergence curves of the same sequences using Algorithm~\ref{algorithm2}, and Fig.~\ref{fig4} verifies that trajectories generated by the model-free learned gains closely align with expert demonstrations.
\begin{figure}[!t]
	\centering
	\begin{subfigure}{0.22\textwidth}
		\centering
		\includegraphics[height=2cm,width=4cm]{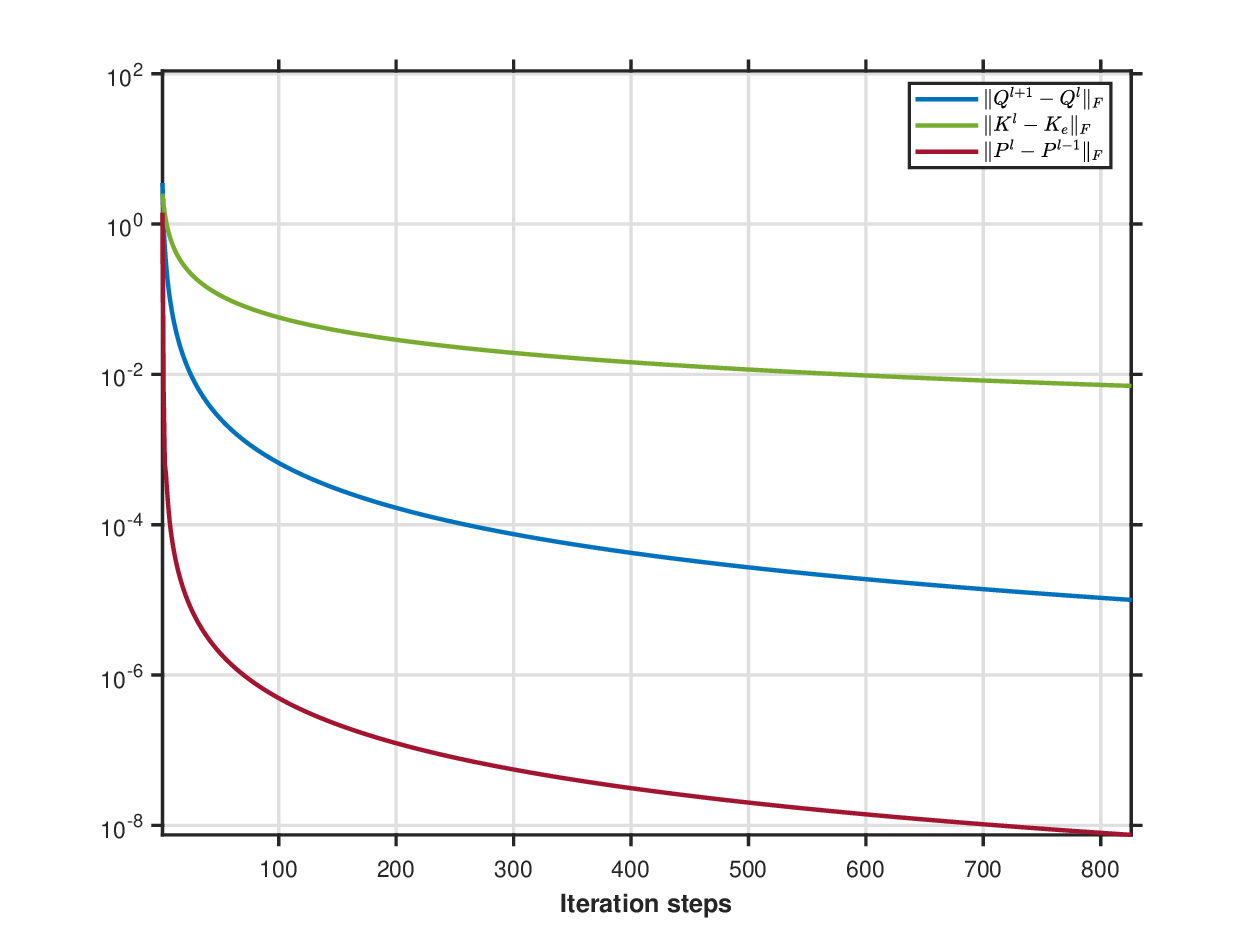} 
		\label{fig3:a}
	\end{subfigure}
	\hfill 
	\begin{subfigure}{0.22\textwidth}
		\centering
		\includegraphics[height=2cm,width=4cm]{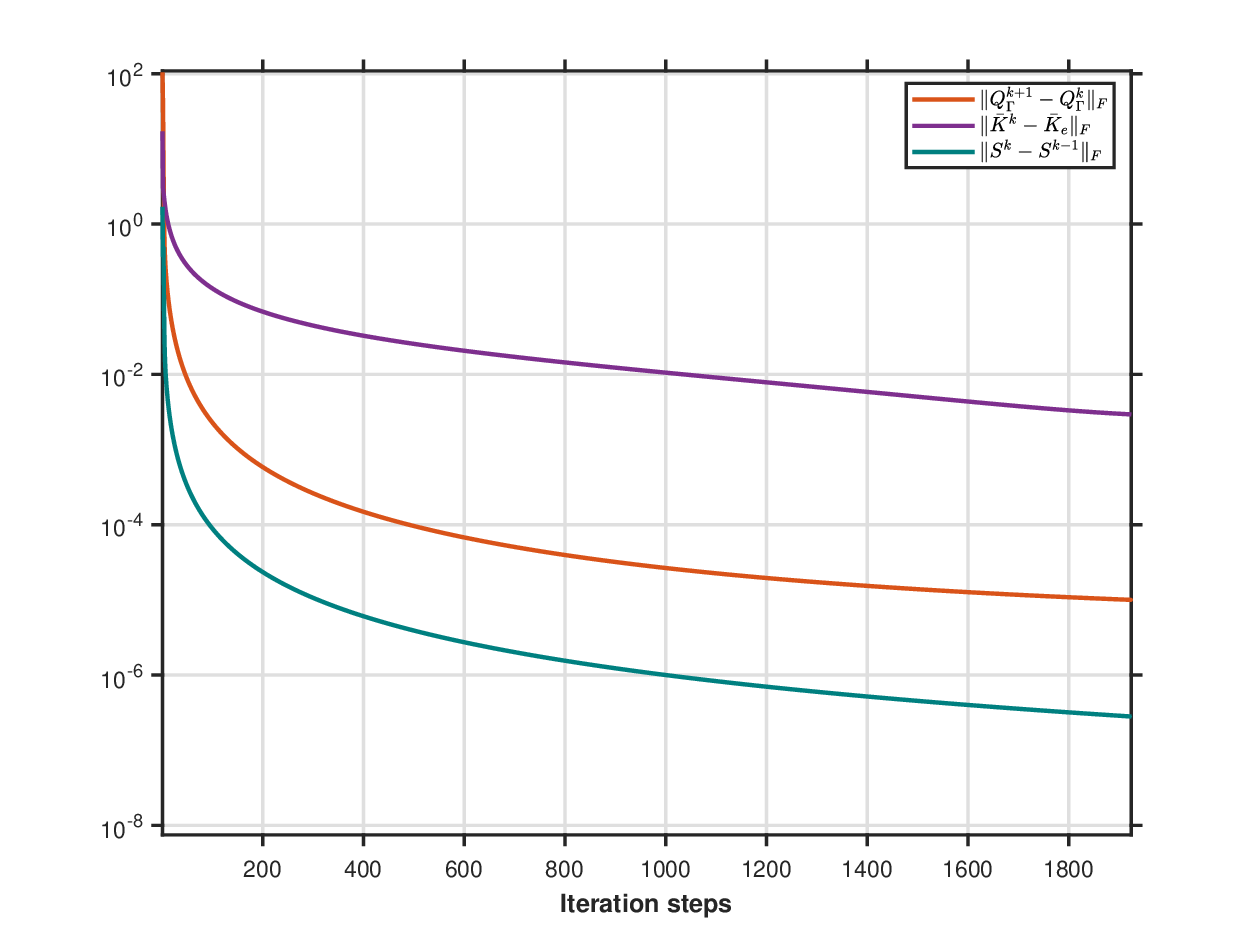} 
		\label{fig3:b}
	\end{subfigure}
    \vspace{-4pt}
	\caption{Convergence of $Q^l$, $K^l$, $P^l$, $Q_{\Gamma}^k$, $\bar{K}^k$ and $S^k$ using Algorithm~\ref{algorithm2}.}
	\label{fig3}
\end{figure}
\begin{figure}[!t]
	\centering
	\begin{subfigure}{0.22\textwidth}
		\centering
		\includegraphics[height=2cm,width=4cm]{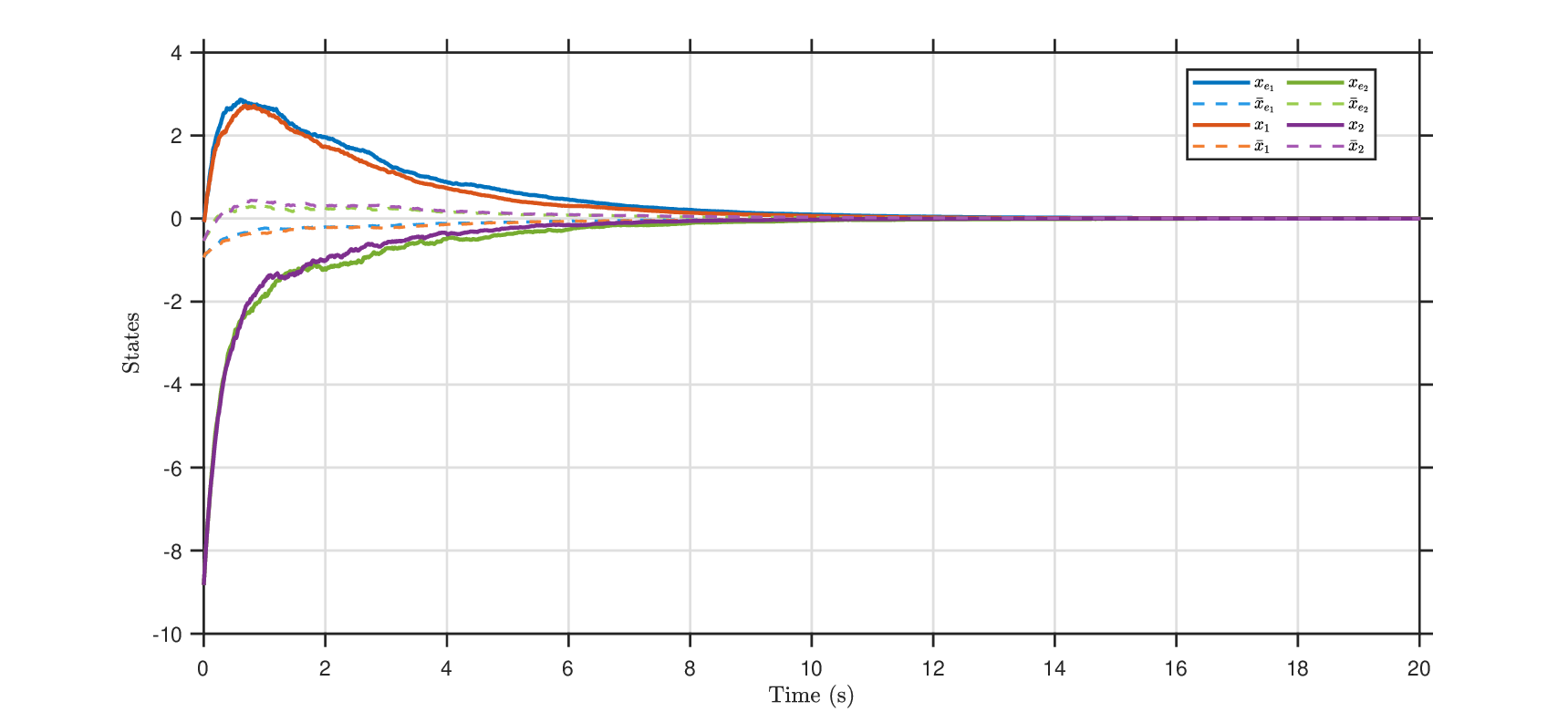} 
		\label{fig4:a}
	\end{subfigure}
	\hfill 
	\begin{subfigure}{0.22\textwidth}
		\centering
		\includegraphics[height=2cm,width=4cm]{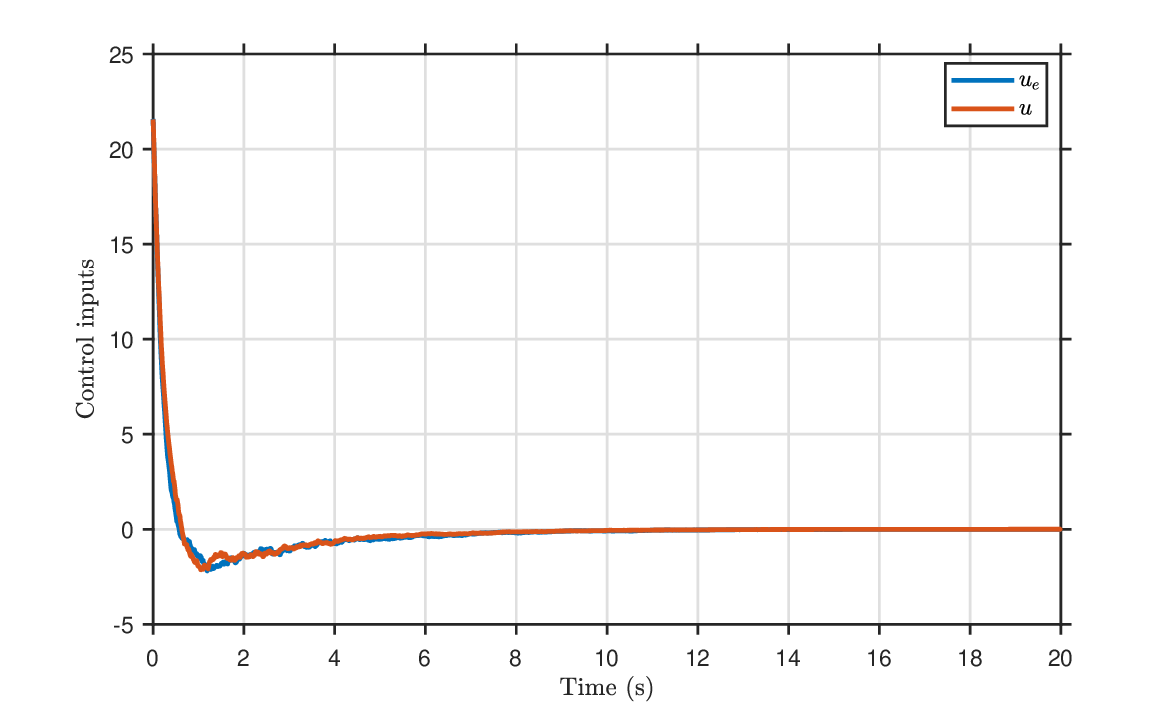} 
		\label{fig4:b}
	\end{subfigure}
    \vspace{-4pt}
	\caption{Trajectory imitation performance of the learner using Algorithm~\ref{algorithm2}.}
	\label{fig4}
\end{figure}

Then we set $R=0.5,1,1.5$ for Algorithm~\ref{algorithm2} to show the non-uniqueness of recovered weights. Despite varying $R$, the learned gains $K^*$ and $\bar{K}^*$ closely match the expert counterparts. Table \ref{tab1} summarizes the converged values $P^*,S^*,Q^*,Q_{\Gamma}^*$ under different $R$.
\begin{table}[H]
	\centering
	\renewcommand{\arraystretch}{1.1}
	\caption{Converged values of Algorithm~\ref{algorithm2} under different $R$.}
	\label{tab1}
	\begin{tabular}{lrrr}
		\hline
		Parameter & $R=0.5$ & $R=1$ & $R=1.5$ \\
		\hline
		$[P^*]_{11}$        & $-0.0001$ & $-0.0074$ & $ 0.0142$ \\
		$[P^*]_{12}$        & $-0.0012$ & $ 0.0096$ & $-0.0261$ \\
		$[P^*]_{22}$        & $ 0.0023$ & $-0.0134$ & $ 0.0399$ \\
		\hline
		$[S^*]_{11}$        & $-0.0080$ & $-0.0150$ & $-0.0342$ \\
		$[S^*]_{12}$        & $ 0.0111$ & $ 0.0225$ & $ 0.0418$ \\
		$[S^*]_{22}$        & $-0.0156$ & $-0.0324$ & $-0.0562$ \\
		\hline
		$[Q^*]_{11}$        & $ 0.8085$ & $ 1.5428$ & $ 3.3351$ \\
		$[Q^*]_{12}$        & $ 1.1232$ & $ 2.2121$ & $ 3.1947$ \\
		$[Q^*]_{22}$        & $ 1.0729$ & $ 2.1579$ & $ 3.3096$ \\
		\hline
		$[Q_{\Gamma}^*]_{11}$ & $-0.1751$ & $-0.1005$ & $-0.2703$ \\
		$[Q_{\Gamma}^*]_{12}$ & $-0.0008$ & $-0.1148$ & $-0.1205$ \\
		$[Q_{\Gamma}^*]_{22}$ & $ 0.2948$ & $ 0.6117$ & $ 0.9184$ \\
		\hline
	\end{tabular}
\end{table}

\textbf{Example 2.} This work \cite{Xu2025} develops a model-free RL method for MF social optimization. We adopt its simulation example as an expert population system, whose system matrices are
\vspace{-3em}
\begin{small}
\begin{align*}
	\renewcommand{\arraystretch}{0.98} 
	A=\begin{bmatrix}
		0.3 & 0.7 \\
		-0.9 & 0.5
	\end{bmatrix},
	B=\begin{bmatrix}
		0.2 \\
		0
	\end{bmatrix},
    C=\begin{bmatrix}
		0.05 & 0.03 \\
		0.05 & 0.02
	\end{bmatrix},
	D=\begin{bmatrix}
		0.05 \\
		0.06
	\end{bmatrix}.
\end{align*}
\end{small}
The coefficients of cost functions are
\vspace{-1em}
\begin{align*}
	Q_e&=\begin{bmatrix}
		3 & 0\\
		0 & 2
	\end{bmatrix},\ \Gamma_e=\begin{bmatrix}
		0.9 & 0\\
		0 & 0.9
	\end{bmatrix},\ R_e=1.25.
\end{align*}
The control gains estimated by the RL method are $K_e=[8.4670\ \: -4.9231]$ and $\bar{K}_e=[-0.4899\ \: 0.4977]$, and the estimated solutions are 
\vspace{-1em}
\begin{align*}
	P_e&=\begin{bmatrix}
		61.800 &  -36.598\\
		-36.598 &   84.241
	\end{bmatrix},\ S_e=\begin{bmatrix}
		-3.5591 &   3.6498\\
		3.6498 &  -9.8665
	\end{bmatrix}.
\end{align*}
Our Algorithm~\ref{algorithm2} does not require prior knowledge of cost weights. We select $R = 1$ and recover the unknown weight matrices only from observed expert trajectories. The learned gains are $K^*=[8.4536\ \: -4.9267]$ and $\bar{K}^*=[-0.4954\ \: 0.4916]$. The converged values $P^*,S^*,Q^*,Q_{\Gamma}^*$ are given as 
\vspace{-1em}
\begin{align*}
	P^*&=\begin{bmatrix}
		0.2192 &  -0.1857\\
		-0.1857 &   0.1500
	\end{bmatrix},\ S^*=\begin{bmatrix}
		-0.0997 &   0.0863\\
		0.0863 &  -0.0721
	\end{bmatrix},\\
    Q^*&=\begin{bmatrix}
    	56.326 & -32.917\\
    	-32.917 &  19.389
    \end{bmatrix}, Q_\Gamma^*=\begin{bmatrix}
    	6.0488  & -4.8727\\
    	-4.8727  &  3.6302
    \end{bmatrix}.
\end{align*}
Note that the solutions $P^*$ and $S^*$ are orders of magnitude smaller than $P_e$ and $S_e$ reported in \cite{Xu2025}, whereas the control gains remain nearly identical. This not only demonstrates the non-uniqueness of inverse RL solutions, but also shows that our algorithm converges to substantially more compact Riccati solutions under an alternative cost parameterization. 

\section{Conclusion}\label{5}
This paper develops a novel inverse RL framework for indefinite LQ MF social control with multiplicative noise. Within this framework, both model-based and model-free algorithms are devised to recover unknown social cost weights and replicate the expert’s optimal control policies. Specifically, the model-free algorithm adopts integral RL to reformulate the model-based iterative equations as purely data-driven updates, relying exclusively on measured system trajectories for implementation. Under proper rank conditions, we provide the proofs of convergence, stability and non-uniqueness of solutions. This framework avoids laborious manual design of objective functions and leverages expert demonstrations to infer collective optimization objectives for stochastic large-scale cooperative systems.

\begin{ack}                               
This work was supported by the National Natural Science Foundation of China under Grants 62573266, 62192753, the National Key R\&D Program of China (2022YFF0712700), the Research Grants Council of Hong Kong under grant 15225124, and PolyU 1-ZVXA, and 4-ZZLT.
\end{ack}

\bibliographystyle{apalike2}   
\bibliography{autosam}

\end{document}